\documentclass[journal,twoside,web]{ieeecolor}
\usepackage{generic}
\usepackage{amsmath,amssymb,amsfonts}
\usepackage{mathrsfs}
\usepackage{graphicx}
\usepackage{algorithm,algorithmic}
\usepackage{hyperref}
\hypersetup{hidelinks=true}
\usepackage{glossaries}
\usepackage{caption}
\usepackage{subcaption}
\usepackage{siunitx}
\usepackage[style=ieee, backend=bibtex]{biblatex}
\usepackage{booktabs}
\usepackage{tikz}
\usetikzlibrary{arrows,shapes,positioning,fit,calc}

\definecolor{mygreen}{rgb}{0,0.6,0}
\definecolor{mygray}{rgb}{0.5,0.5,0.5}
\definecolor{mymauve}{rgb}{0.58,0,0.82}

\def\BibTeX{{\rm B\kern-.05em{\sc i\kern-.025em b}\kern-.08em
    T\kern-.1667em\lower.7ex\hbox{E}\kern-.125emX}}
\newcommand{\R}{\mathbb{R}}
\newcommand{\N}{\mathbb{R}}
\newcommand{\col}[1]{\text{col}\{#1\}}
\newcommand{\diag}[1]{\text{diag}\{#1\}}
\newcommand{\mat}[1]{\begin{bmatrix}#1\end{bmatrix}}
\newcommand{\trans}{\top}
\newcommand{\lipIdx}{M}
\newcommand{\proj}[2]{\text{proj}_{#2}(#1)}
\newcommand{\vproj}[3]{\Gamma_{#3}(#1,#2)}
\newcommand{\npd}[1]{\dot{\overline{#1}}}
\newcommand{\eigmin}[1]{\text{eigmin}\{#1\}}
\newcommand{\eigmax}[1]{\text{eigmax}\{#1\}}
\newcommand{\inter}[1]{\text{int}\{#1\}}
\newcommand{\bound}[1]{\text{bnd}\{#1\}}

\newcommand{\convIdx}{m}
\newcommand{\numOfIneq}[1]{\ifthenelse{\equal{#1}{}}{C_{\text{in}#1}}{C_{\text{in},#1}}}
\newcommand{\numOfEq}[1]{\ifthenelse{\equal{#1}{}}{C_{\text{eq}#1}}{C_{\text{eq},#1}}}
\newcommand{\setOfIneq}[1]{\ifthenelse{\equal{#1}{}}{\mathcal{C}_{\text{in}#1}}{\mathcal{C}_{\text{in},#1}}}
\newcommand{\setOfEq}[1]{\ifthenelse{\equal{#1}{}}{\mathcal{C}_{\text{eq}#1}}{\mathcal{C}_{\text{eq},#1}}}
\newcommand{\yrm}{\mathrm{y}}

\newcommand{\inci}{E}

\newcommand{\numOfAgents}{N}
\newcommand{\setOfAgents}{\mathcal{A}}
\newcommand{\numOfControllers}{K}
\newcommand{\setOfControllers}{\mathcal{C}}

\newcommand{\rl}[1]{\bar{#1}}
\newcommand{\err}[1]{\tilde{#1}}
\newcommand{\opt}[1]{{#1}^*}
\newcommand{\xat}[2]{\ifthenelse{\equal{#2}{}}{\x_{#1}}{\x_{#1}(#2)}}
\newcommand{\uat}[2]{\ifthenelse{\equal{#2}{}}{\mv{u}_{#1}}{\mv{u}_{#1}(#2)}}
\newcommand{\setEq}{\mathcal{E}_\text{eq}}
\newcommand{\isoseip}{IF-OFEIP}

\newcommand{\ofpIdx}{\rho}
\newcommand{\ifpIdx}{\nu}
\newcommand{\eqM}{\mathcal{M}_\text{eq}}
\newcommand{\feedt}{\gamma}

\newacronym{eip}{EIP}{equilibrium-independent passive}
\newacronym{oseip}{OFEIP}{output feedback EIP}
\newacronym{iseip}{IFEIP}{input feedforward EIP}
\newacronym{eio}{EIO}{equilibrium-independent observable}

\newtheorem{definition}{Definition}
\newtheorem{theorem}[definition]{Theorem}
\newtheorem{remark}[definition]{Remark}
\newtheorem{proposition}[definition]{Proposition}
\newtheorem{assumption}[definition]{Assumption}
\newtheorem{corollary}[definition]{Corollary}
\newtheorem{example}{Example}

\begin{document}
\title{Passivity-Based Local Design Conditions for Global Optimality in Distributed Convex Optimization}
\author{Pol Jane-Soneira, 
Charles Muller, Felix Strehle and Sören Hohmann, \IEEEmembership{Member, IEEE}
\thanks{This paragraph of the first footnote will contain the date on 
which you submitted your paper for review. It will also contain support 
information, including sponsor and financial support acknowledgment.  }
\thanks{The authors are with the Institute of Control Systems, Karlsruher Institute of Technology, 76131 Karlsruhe, Germany (e-mails: \{pol.jane,felix.strehle, soeren.hohmann\}@kit.edu, charles.muller@student.kit.edu). Corresponding author is Pol Jane-Soneira.}}

\maketitle

\begin{abstract}
	In recent times, various distributed optimization algorithms have been proposed for whose specific agent dynamics global optimality and convergence is proven. 
	However, there exist no general conditions for the design of such algorithms. 
	In this paper, we leverage passivity theory to first establish a distributed optimization framework with local design requirements for the agent dynamics in both unconstrained and constrained problems with undirected communication topologies.
	Under the roof of these requirements, the agents may use heterogeneous optimization algorithms without compromising global optimality and convergence.
	Subsequently, we propose some exemplary agent systems that comply with the established requirements.	
	Compared to existing approaches, our algorithms do not require any global initialization nor communication of multiple variables. 
	Consequently, the agents may leave or rejoin the networked optimization without compromising convergence to the correct global optimizer.    
    Furthermore, we show that for unconstrained optimization, an extension to directed communication topologies is possible.
    Simulation results illustrate the plug-and-play capabilities and interoperability of the proposed agent dynamics.
\end{abstract}

\begin{IEEEkeywords}
Distributed optimization, passivity-based control, networked control, convex optimization
\end{IEEEkeywords}

\section{Introduction}
\label{sec:introduction}
\IEEEPARstart{D}{istributed} optimization has attracted increasing attention in a wide range of applications such as wireless networks, machine learning, energy systems, or distributed parameter estimation~\cite{boyd2011distributed,bazerque2009distributed,kekatos2012distributed,primadianto2016review,wang2017distributed,molzahn2017survey}. Distributed optimization originated first from Dual Decomposition in applications concerning operations research~\cite{boyd2011distributed,dantzig1960decomposition}. In the last decade, numerous algorithms have been proposed, both in continuous \cite{kia2015distributed,hatanaka2018passivity,li2020input,notarnicola2023gradient,esteki2023distributed} and discrete-time \cite{nedic2009distributed,zhu2011distributed, shi2015extra, nedic2017achieving, lei2016primal,varagnolo2015newton} (see~\cite{yang2019survey} for a survey), many of them also borrowing ideas from control theory. 

In distributed optimization algorithms, different agents seek to minimize the sum of all private objective functions cooperatively by performing local computations and only exchanging information with their neighbors. 
The algorithms are typically designed such that, provided that all agents follow the proposed, specific algorithm, global optimality and convergence are achieved. 


The first control-theoretic approach to distributed optimization arose from combining a simple gradient descent algorithm with distributed consensus-based algorithms inducing a distributed proportional-integral action~\cite{wang2010control}, which was later named distributed PI algorithm~\cite{yang2019survey}. In the original version, the agents communicate two variables over an undirected graph. Since then, several extensions of this algorithm have been proposed. The authors in~\cite{gharesifard2013distributed} show that an immediate extension to directed communication topologies is not possible and thus propose a variant capable of handling a directed graph as communication network and non-differentiable objective functions. In~\cite{hatanaka2018passivity}, the distributed PI algorithm is analyzed with passivity theory in order to extend the framework to communication delays and constraints. In~\cite{lei2016primal}, the distributed PI algorithm is cast in a discrete-time setting and extended to handle compact, convex constraints. 
The authors in~\cite{kia2015distributed} propose a modification such that, in contrast to the original two variables, the agents exchange only a single variable in order to achieve coordination. As a consequence, the algorithm has enhanced privacy properties. However, this comes at the cost of imposing a global initialization condition for all agents. A global initialization is inappropriate for distributed optimization in large networks, since it has to be performed every time the optimization setup changes. This practically impedes agents leaving or (re)joining the network optimization, which is crucial in large scale, flexible networks as found in many applications~\cite{bazerque2009distributed,kekatos2012distributed,primadianto2016review,wang2017distributed,molzahn2017survey}. 
Furthermore, the authors in~\cite{kia2015distributed} develop a framework based on Lyapunov theory in order to allow discrete-time communication. Extensions to the powerful method presented in~\cite{kia2015distributed} have been proposed, transferring it to a discrete-time setting in~\cite{yao2018distributed}, or analyzing convergence with passivity theory and time-varying parameters~\cite{li2020input, li2020distributed}. 
In~\cite{touri2023unified}, the authors present an observer-based approach and show that many of the algorithms in the literature are a special case of the method they propose. 

Parallel to that line of research, other distributed optimization algorithms without an evident control-theoretic interpretation have been presented. The EXTRA algorithm presented in~\cite{shi2015extra} is shown to converge at a linear rate when the objective functions of the agents are strongly convex. Furthermore, it is shown in~\cite{yao2018distributed}, that the discrete-time version of the distributed PI algorithm is a special case of the EXTRA algorithm when the mixing matrices are properly chosen. The convergence of the EXTRA algorithm is ensured independently of the network topology and the agents only need to exchange a single variable. In~\cite{xi_DEXTRAFast_2017}, the EXTRA algorithm is extended to handle communication over directed graphs. However, similarly as in the case of the distributed PI, a global initialization step is necessary in the EXTRA and all its variations. Another algorithm is the DIGing, proposed in~\cite{nedic2017achieving, qu2017harnessing}, which tracks the average gradient by using dynamic average consensus. The algorithm only requires a special local initialization, but, similarly to the original version of the distributed PI in~\cite{wang2010control}, the agents need to exchange two variables. Recently, a control-theoretic perspective for the DIGing algorithm has been presented~\cite{notarnicola2023gradient}. The authors construct the DIGing algorithm by adding a distributed servomechanism to compensate the steady-state error of decentralized gradient descent. Exploiting the control-theoretic perspective, exponential stability is proven by means of Lyapunov theory. Nevertheless, either the global initialization or the necessity of exchanging two variables still remains in all existing methods. 

Another class of distributed optimization algorithms exploits second-order derivatives and solves an approximate Newton-Raphson algorithm in a distributed manner by means of average consensus \cite{varagnolo2015newton, zanella2011newton, bof2018multiagent, moradian2022distributed}. It was first proposed in~\cite{zanella2011newton,varagnolo2015newton} and introduces the idea of tracking the gradient of the global cost function, serving as precedent of the DIGing algorithms discussed above. It has been extended to lossy communication networks in~\cite{bof2018multiagent}, and modified in~\cite{moradian2022distributed} in order to achieve a comparable convergence rate to the centralized Newton-Raphson algorithm while using less communication and storage compared to existing methods. Albeit powerful, these methods require exchanging multiple additional variables and, e.g.~\cite{moradian2022distributed}, also an additional, special global initialization.


All in all, in recent years, very powerful algorithms for distributed optimization have been proposed both from the mathematical and control theoretic perspective. 
However, up to now, there exist no conditions that enable the local, independent design of such algorithms and ensure their interoperability. 
All the existing works only propose specific algorithms and, provided that all agents follow their specific algorithm, derive conditions on the tuning parameters for achieving global optimality and convergence. 

\emph{Contributions:} In sharp contrast to existing work, instead of proposing a specific algorithm, we change the perspective and first propose local design requirements for the agent-specific variable updating rules ensuring global optimality and convergence to the correct global optimizer. For that, we leverage the classic passivity-based control structure for network flow problems~\cite{arcak2016networks,burger2014duality} and power systems~\cite{strehle2021unified} to propose a distributed optimization framework.
Following these design requirements, the agents do not need further coordination, can freely choose their optimization dynamics, and can assess the global optimality and convergence with modular, local conditions. This makes our framework easily scalable and enables a plug-and-play operation of agents without compromising global optimality and convergence.
Using these design requirements, we then propose algorithms to address the shortcomings in the literature, i.e., the need for either a global initialization step or the communication of multiple variables in order to achieve global optimality and convergence.  

This work is structured as follows. In Section~\ref{sec:problem_formulation_preliminaries}, the problem formulation is given and mathematical as well as system theoretical preliminaries are presented. In Section~\ref{sec:distributed_opt}, the distributed optimization framework is presented and the local design requirements for achieving global optimality and convergence are derived. Specific agent and controller dynamics complying with the design requirements are presented in Section~\ref{sec:agent_controller_dynamics}.
In Section~\ref{sec:convergence_undirected}, less restrictive convergence results are given considering the proposed agent dynamics. Directed communication topologies are considered in Section~\ref{sec:convergence_directed}. In Section~\ref{sec:results}, an illustrative example is presented, and Section~\ref{sec:conclusion} concludes this work.

\emph{Notation:} The transpose of a vector $x \in \R^{n}$ is written as $x^\trans$. The vector $x = \col{x_i}$ and matrix $X = \diag{x_i}$ are the $n\times 1$ column vector and $n \times n$ diagonal matrix of the elements $x_i$, $i = 1,\dots, n$, respectively. Let $I_{n}$ denote the $n \times n$ identity matrix, and $1_n \in \R^n$ a vector of ones. 
The maximum and minimum eigenvalues of $X$ are denoted as $\eigmax{X}$ and $\eigmin{X}$. 
Calligraphic letters $\mathcal{A}$ represent sets or graphs. 
A directed graph is denoted by $\mathcal{G}(\setOfAgents,\mathcal{E})$, where $\setOfAgents$ is the set of agents and $\mathcal{E} \subset \setOfAgents\times \setOfAgents$ the set of edges. The incidence matrix $\inci \in \R^{\vert \setOfAgents \vert \times \vert \mathcal{E} \vert }$ is defined as $\inci = (m_{ij})$ with $m_{ij} = -1$ if edge $e_j \in \mathcal{E}$ leaves node $v_i \in \setOfAgents$, $m_{ij} = 1$ if edge $e_j \in \mathcal{E}$ enters node $v_i \in \setOfAgents$, and $m_{ij} = 0$ otherwise. The Kronecker product of $x$ and $y$ is denoted by $x \otimes y$. 
The gradient $\nabla_x f: \R^n \rightarrow \R^n$ of a scalar function $f:\R^n \rightarrow \R$ is defined as the column vector. The subscript indicates with respect to (w.r.t.) which variables we differentiate if the function has multiple variables. For convenience, we define $\err{f}(x,y) = f(x) - f(y)$.

\section{Problem Formulation and Preliminaries}  \label{sec:problem_formulation_preliminaries}

Before stating the problem, we introduce some basic concepts from convex analysis and system theory that are necessary for the further developments. 
\subsection{Convex Analysis}

A differentiable function $f:\R^n \rightarrow \R $ is convex over a convex set $\mathcal{X} \subseteq \R^n$ iff 
\begin{align} \label{eq:convexity_differentiable}
    \left( \nabla f(y) - \nabla f(x) \right)^\trans (y - x) \geq m (y-x)^\trans (y-x)
\end{align}
holds $\forall x, y \in \mathcal{X}$ and $\convIdx = 0$. If~\eqref{eq:convexity_differentiable} holds strictly with $\convIdx=0$, the function $f$ is called strictly convex, and if~\eqref{eq:convexity_differentiable} holds for $\convIdx>0$, it is called $\convIdx$-strongly convex. 
Furthermore, for any convex function $f$, the inequalities
\begin{equation} \label{eq:convexity_inequalities}
    \begin{array}{cc}
        f(x_1) - f(x_2) \leq \nabla f(x_1) (x_1 -x_2) \\ 
        f(x_1) - f(x_2) \geq \nabla f(x_2) (x_1 -x_2) 
    \end{array}
\end{equation}
hold. We say a function $f:\R^n \rightarrow \R$ is $\lipIdx$-Lipschitz, if $\Vert f(x) -  f(y) \Vert \leq \lipIdx \Vert x - y \Vert$ holds with $\lipIdx \in \R_{\geq 0}$. 

\subsection{Passivity Theory}

Consider the nonlinear input-affine system
\begin{subequations} \label{eq:sys}
    \begin{align} \label{eq:sys_dynamics}
        \dot{x} &= F(x) + Bu \\
        y &= Cx + Du \label{eq:sys_output},
    \end{align}
\end{subequations}
with $x \in \R^n$, $u \in \R^p$, $F: \R^n \rightarrow \R^n$ and matrices $B,C$ and $D$ of appropriate dimension.

\begin{definition} \label{def:integral_action}
    We say that system~\eqref{eq:sys} is a system with integral action, if there exist no other input leading to an equilibrium other than $\rl{u}=0$. 
\end{definition}

\begin{definition}[{\cite[Def.~1.1]{arcak2016networks}}]
    System~\eqref{eq:sys} is dissipative w.r.t.\ the supply rate $w: \R^n \times \R^p \rightarrow \R$, if there exists a continuously differentiable storage function $V: \R^n \rightarrow \R$ with $V(0) = 0$ and $V(x)\geq 0$ such that
    \begin{equation}
        \dot{V} = \nabla V^\trans \left( F(x) + Bu \right) \leq w(u,y).
    \end{equation}
    Furthermore, a system is called passive, if it is dissipative with the supply rate $w(u,y) = y^\top u$. 
\end{definition}

Now suppose there exists a set of equilibria $\setEq$, where, for every $\rl{x} \in \setEq$, there is a unique $\rl{u} \in \R^p$ satisfying $0 = F(\rl{x}) + B \rl{u}$. 

\begin{definition}[{\cite[Def.~3.1]{arcak2016networks}},{\cite[Def.~3.2]{simpson2018equilibrium}}]
    The system~\eqref{eq:sys} is equilibrium-independent dissipative w.r.t.\ the supply rate $w$, if there exists a continuously differentiable storage function $V: \R^n \times \setEq \rightarrow \R$ satisfying, for all $(x, \rl{x}, u) \in \R^n \times \setEq \times \R^p$, $V(\rl{x}, \rl{x}) = 0$, $V(x, \rl{x}) \geq 0$, and
    \begin{equation}
        \nabla_x V \left( F(x) + Bu \right) \leq w(u-\rl{u}, y-\rl{y}). 
    \end{equation}
\end{definition}

For convenience, define the error variables $\err{x} = x - \rl{x}$, $\err{u} = u - \rl{u}$ and $\err{y} = y - \rl{y}$. 

\begin{definition}
    System~\eqref{eq:sys} is \gls{eip}, if it is equilibrium-independent dissipative w.r.t.\ the supply rate $w(\err{u},\err{y}) = \err{u}^\trans \err{y} - \ofpIdx \err{y}^\trans \err{y} - \ifpIdx \err{u}^\trans \err{u} - \Psi(\err{x})$, where $\Psi: \R^n \rightarrow \R_{\geq 0}$ is a positive semidefinite function and $\ofpIdx,\ifpIdx\geq 0$. Moreover, the system~\eqref{eq:sys} is called
    \begin{itemize}
        \item \gls{oseip} with index $\ofpIdx$, if $\ofpIdx\in\R$, 
        \item \gls{iseip} with index $\ifpIdx$, if $\ifpIdx \in \R$, 
        \item strictly \gls{eip}, if $\Psi(\err{x})$ is positive definite.  
    \end{itemize}
    In short, we say a system is \gls{oseip}($\ofpIdx$) or \gls{iseip}($\ifpIdx$) in the remainder of the paper. If a system is simultaneously \gls{oseip}$(\ofpIdx)$ and \gls{iseip}$(\ifpIdx)$, we call it \isoseip$(\ifpIdx,\ofpIdx)$. If $\ifpIdx$, $\ofpIdx > 0$, we say the system has an excess of passivity. If the indices are negative, we say the system has a lack of passivity. 
\end{definition}

\begin{definition}[{\cite[Def.~4.1]{simpson2018equilibrium}}]
    The system~\eqref{eq:sys} is \gls{eio}, if for every $\rl{x} \in \setEq$ with associated equilibrium input $\rl{u}$ and output $\rl{y}$, no trajectory of~\eqref{eq:sys} can remain within the set $\left\{ x \in \R \mid C x + D \rl{u} = \rl{y} \right\}$ other than the equilibrium trajectory $x(t) = \rl{x}$. 
\end{definition}

\subsection{Problem Formulation and Main Idea}

Consider a set of agents $\setOfAgents = \{1,\dots,\numOfAgents \}$, $\numOfAgents \in \N$. We define the network optimization problem 
\begin{subequations} \label{eq:optimization_problem}
    \begin{align}
        \min & \sum_{i \in \setOfAgents} f_i(\yrm) \\
        \text{s.t. } & g_{il}(\yrm) \leq 0, \quad \forall l \in \setOfIneq{i}, \, \forall i \in \setOfAgents \label{eq:inequality_constraints} \\
        & h_{ij}(\yrm)  = 0, \quad \forall j \in \setOfEq{i}, \, \forall i \in \setOfAgents \label{eq:equality_constraints}
    \end{align}
\end{subequations}
where $\yrm \in \R^n$. The objective function $f_i: \R^n \rightarrow \R$ and the constraints $g_{il}: \R^n \rightarrow \R$ and $h_{ij}: \R^n \rightarrow \R$ are private to agent $i \in \setOfAgents$. The variable $\yrm \in \R^n$ is a global variable the agents. The sets $\setOfIneq{i} =\{1,\dots,\numOfIneq{i} \}$ and $\setOfEq{i} = \{1,\dots,\numOfEq{i}\}$ are the sets of inequality and equality constraints of agent $i\in\setOfAgents$. The functions $g_{il}: \R^{n} \rightarrow \R$ are assumed to be convex, the functions $h_{ij}: \R^{n} \rightarrow \R$ to be affine, i.e., $h_{ij}(\yrm) = a_{ij}^\trans \yrm + c_{ij}$. Further, we assume that~\eqref{eq:optimization_problem} is strictly convex and Slater's condition holds, i.e., there exists an $\yrm \in \R^n$ where constraints~\eqref{eq:equality_constraints} hold and constraints~\eqref{eq:inequality_constraints} are strictly feasible~\cite[p.~244]{boyd2004convex}. Note that these are standard assumptions in convex optimization~\cite{hatanaka2018passivity,lei2016primal,boyd2004convex}. 

The problem considered in this work is the agents cooperatively finding a solution to the network-wide optimization problem~\eqref{eq:optimization_problem} only using local information and communicating with the neighbors. For that, each agent $i \in \setOfAgents$ typically holds a local estimate $y_i\in \R^n$ of the global optimizer $\opt{\yrm} \in \R^n$. The agents can perform local computations with its private functions and communicate the local estimate $y_i$ with the neighbors. The communication topology is described by an undirected graph $\mathcal{G}(\setOfAgents,\mathcal{E})$, where the agents represent the nodes and an edge $e_{ij} \in \mathcal{E} \subset \setOfAgents \times \setOfAgents$ between agent $i$ and $j$ means that the agents are neighbors and can exchange their local estimations. 

\emph{Main idea:} In this work, we propose a class of algorithms that decompose into the form shown in Figure~\ref{fig:burger_structure}. It consists of agent systems $\Sigma_i$, $i \in \setOfAgents$ and controller systems $\Pi_k$, $k\in\setOfControllers$, where $\setOfControllers= \{1,\dots,\numOfControllers \}$ is the set of controllers. Each controller is associated with an edge of graph $\mathcal{G}$, i.e., with the communication between two agents, and thus $\vert \setOfControllers \vert = \vert \mathcal{E} \vert$.  
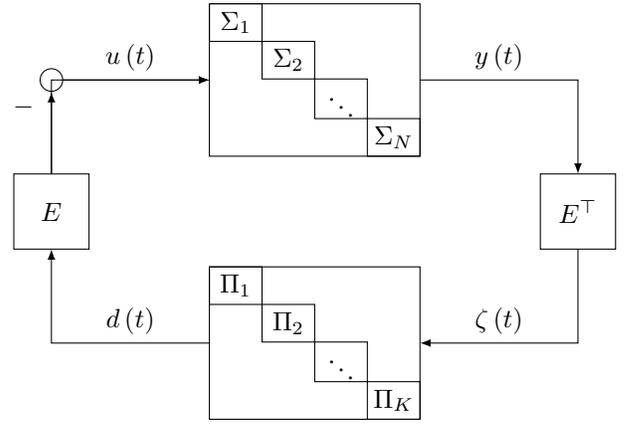
\begin{figure}[t]
	\centering
    \begin{tikzpicture}
	\begin{scope}[node distance=0, every node/.style={rectangle,draw=black}, minimum height=5mm, minimum width=7mm, outer sep=0pt, inner sep=0pt, on grid]
		\node (s1) at (0,0) {$\Sigma_1$};
		\node[below right=of s1.south east](s2){$\Sigma_2$};
		\node[below right=of s2.south east](sSpace){$\ddots$};
		\node[below right=of sSpace.south east](sN){$\Sigma_N$};
		\node[fit=(s1) (sN)] (s) {};
		\node[below=3.5cm of s1] (p1){$\Pi_1$};
		\node[below right=of p1.south east](p2){$\Pi_2$};
		\node[below right=of p2.south east](pSpace){$\ddots$};
		\node[below right=of pSpace.south east](pN){$\Pi_K$};
		\node[fit=(p1) (pN)] (p) {}; 
		\node[minimum height=10mm,minimum width=10mm,left=3cm of {$(p)!0.5!(s)$}] (E) {${E}$};
		\node[minimum height=10mm,minimum width=10mm,right=3cm of {$(p)!0.5!(s)$}] (ET) {${E}^\top$};
	\end{scope}
	
	\node[circle,draw,minimum height=3mm,minimum width=3mm, inner sep=0pt](sum) at (E |- s) {};
	\node[below left=0mm and 0mm of sum] {$-$};
	
	\begin{scope}[-latex, every node/.style={above}]
		\draw[-latex] (p) -| node[near start] {${d}\left(t\right)$} (E);
		\draw[-latex] (s) -| node[near start] {${y}\left(t\right)$} (ET);
		\draw[-latex] (ET) |- node[near end] {${\zeta}\left(t\right)$} (p);
		\draw[-latex] (E) -- (sum);
		\draw[-latex] (sum) -- (s); 
		\draw[draw opacity=0] (E) |- node[draw opacity=100,near end] {${u}\left(t\right)$} (s);
	\end{scope}
\end{tikzpicture}
    \caption{Block diagram of the proposed algorithm structure, composed of agent systems $\Sigma_i$ and controller systems $\Pi_k$.}
	\label{fig:burger_structure}
\end{figure} 
The agent systems
\begin{align} \label{eq:agents}
    \Sigma_i: 
    \begin{array}{cl}
        \dot{x}_i & = \chi_i(x_i, u_i)\\
        y_i & = \psi_i(x_i, u_i),
    \end{array}
\end{align}
with $\chi_i:\R^n \times \R^n \rightarrow \R^n$ and $\psi_i:\R^n\times \R^n \rightarrow \R^n$, hold the state $x_i \in \R^n$, which represents an intermediate value for the estimation $y_i \in \R^n$ of the global optimizer $\opt{\yrm}$. The input of the agent systems is $u_i \in \R^n$, which should be chosen such that the output $y_i \in \R^n$ of the agents converges to the global optimizer $\opt{\yrm}$. The controller systems are described by
\begin{align} \label{eq:controllers}
    \Pi_k: 
    \begin{array}{cl}
        \dot{z}_k & = \Phi_k(z_k, \zeta_k) \\
        d_k & = \sigma_k(z_k),
    \end{array}
\end{align}
with $\Phi_k:\R^n \times \R^n \rightarrow \R^n$, $\sigma_k:\R^n \rightarrow \R^n$, state $z_k \in \R^n$, input $\zeta_k \in \R^n$, and output $d_k \in \R^n$, which is fed back to the agent systems in order to find the global optimizer $\opt{\yrm}$. The agent and controller systems are interconnected by an incidence matrix $\inci$ lifted to the space $\R^{\numOfAgents n}$ of the networked system states as in
\begin{subequations}\label{eq:interconnection}
    \begin{align} \label{eq:interconnection1}
        \zeta &= \left( \inci \otimes I_n \right)^\trans y \\
        u &= - \left( \inci \otimes I_n \right) d, \label{eq:interconnection2}
    \end{align}
\end{subequations}
where $\zeta$, $u$, $d$ and $y$ are the stacked variables of the agents and controllers, i.e., $y = \col{y_i}$, etc.
Taking into account the definition of the incidence matrix~$\inci$, it becomes clear from~\eqref{eq:interconnection1} that the inputs $\zeta_k$ of the controller systems are the differences between the outputs $y_i$ of different agents, i.e., the estimations $y_i$ of different agents of the global optimizer $\opt{\yrm}$. 

With the partitioning of the algorithm into agent and controller systems, the question arises of where to physically implement the controller systems. We show in the next example that this is a purely technical implementation issue which does not limit the findings of this work.

\begin{example}
    Consider a network of four agents interconnected as in Figure~\ref{fig:example1}.
    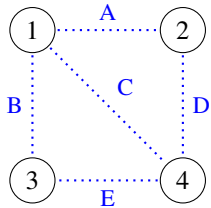
\begin{figure}[ht]
        \centering
        \begin{tikzpicture}
	\begin{scope}[node distance={2cm and 2cm}, every node/.style={circle,draw=black}, minimum size=6mm, outer sep=1pt, inner sep=2pt, on grid]
		\node (n1) at (0,0) {1};
		\node[right=of n1](n2){2};
		\node[below=of n1](n3){3};
		\node[below=of n2](n4){4};
	\end{scope}
	
	\begin{scope}[dotted, thick, blue, every node/.style={black, midway, minimum size=0pt, auto=left}]
		\draw (n1) -- node[blue] {\small A} (n2);
		\draw (n1) --  node[blue, left] {\small B} (n3);
		\draw (n1) -- node[blue] {\small C} (n4);
		\draw (n2) -- node[blue] {\small D} (n4);
		\draw (n4) -- node[blue] {\small E} (n3); 
	\end{scope}
\end{tikzpicture}
        \caption{Exemplary network composed of 4 agents (nodes) and communication links in blue (edges).}
        \label{fig:example1}
    \end{figure}
    The proposed method introduces five controllers A - E, one for each edge. The controller dynamics can be interpreted as independent entities with computation capabilities or, alternatively, as part of the same system implementing the agent dynamics, as shown in Figure~\ref{fig:example1_2}. In the second case, the agents implement the subsystems inside the dashed line in Figure~\ref{fig:withAgents}, composed of the agent dynamics~\eqref{eq:agents} and the controller dynamics~\eqref{eq:controllers}.
    Depending on the interpretation, the agents have to exchange different variables and different information flows occur. When the controllers are independent entities, the agents send their estimation $y_i$ and receive the input $u_i$. When the controllers are implemented together with the agent dynamics, the agents send either their estimation $y_i$ or the controller output $d_k$, depending on whether they have the controller or not (c.f. Figure~\ref{fig:withAgents}). 
    \begin{figure}[h]
        \begin{subfigure}[b]{0.2\textwidth}
            \centering
            \begin{tikzpicture}
	\begin{scope}[node distance={1.4cm and 1.4cm}, every node/.style={draw=black}, minimum size=6mm, outer sep=1pt, inner sep=2pt, on grid]
		\node[circle] (n1) at (0,0) {1};
		\node[diamond, draw=mygreen, right=of n1](nA){A};
		\node[circle, right=of nA](n2){2};
		\node[diamond, draw=mygreen, below=of n1](nB){B};
		\node[diamond, draw=mygreen, below=of nA](nC){C};
		\node[circle, below=of nB](n3){3};
		\node[diamond, draw=mygreen, below=of n2](nD){D};
		\node[circle, below=of nD](n4){4};
		\node[diamond, draw=mygreen, right=of n3](nE){E};
	\end{scope}
	
	\begin{scope}[dotted, thick, blue, every node/.style={scale=0.8, minimum size=0pt, sloped}]
		\draw (n1) -- node[pos=0.5, auto=left, mygreen] {\small$\leftarrow {d}_A $} node[pos=0.45, auto=right] {\small$\rightarrow {y}_1$} (nA);
		\draw (nA) -- node[pos=0.4, auto=left, mygreen] {\small ${d}_A \rightarrow$} node[pos=0.45, auto=right] {\small$\leftarrow {y}_2$} (n2);
		\draw (n1) -- node[pos=0.6, auto=left, mygreen] {\small$\leftarrow{d}_B $} node[pos=0.5, auto=right] {\small${y}_1 \rightarrow$} (nB);
		\draw (nB) -- node[pos=0.4, auto=left, mygreen] {\small${d}_B \rightarrow$} node[pos=0.45, auto=right] {\small$\leftarrow {y}_3$} (n3);
		\draw (n1) -- node[pos=0.7, auto=left, mygreen] {\small$\leftarrow{d}_C $} node[pos=0.7, auto=right] {\small${y}_1 \rightarrow$} (nC);
		\draw (nC) -- node[pos=0.3, auto=left, mygreen] {\small${d}_C \rightarrow$} node[pos=0.3, auto=right] {\small$\leftarrow {y}_4$} (n4);
		\draw (n2) -- node[pos=0.6, auto=left, mygreen] {\small${\leftarrow d}_D$} node[pos=0.6, auto=right] {\small$ {y}_2\rightarrow$} (nD);
		\draw (nD) -- node[pos=0.4, auto=left, mygreen] {\small${d}_D \rightarrow$} node[pos=0.45, auto=right] {\small$\leftarrow {y}_4$} (n4);
		\draw (n4) -- node[pos=0.6, auto=left, mygreen] {\small${d}_E \rightarrow$} node[pos=0.6, auto=right] {\small$\leftarrow {y}_4$} (nE);
		\draw (nE) -- node[pos=0.4, auto=left, mygreen] {\small$\leftarrow{d}_E$} node[pos=0.45, auto=right] {\small${y}_3\rightarrow$} (n3);
	\end{scope}

	\node[draw=white,minimum width = 3.8cm, minimum height=4.25cm] at (1.4,-1.3) {};

\end{tikzpicture}
            \caption{Controllers as independent entities}
            \label{fig:independent}
        \end{subfigure}
        \hfill
        \begin{subfigure}[b]{0.27\textwidth}
            \centering
            \begin{tikzpicture}
	\begin{scope}[node distance={3cm and 3cm}, every node/.style={circle, draw=black}, minimum size=6mm, outer sep=1pt, inner sep=2pt, on grid]
		\node (n1) at (0,0) {1};
		\node[right=of n1](n2){2};
		\node[below=of n1](n3){3};
		\node[below=of n2](n4){4};
	\end{scope}
	
	\begin{scope}[node distance=2pt, every node/.style={diamond, draw=mygreen}, minimum size=7mm, outer sep=0pt, inner sep=2pt]
		\node[right=of n1.east](nA){A};
		\node[below=of n1.south](nB){B};
		\node[below right=of n1.south east](nC){C};
		\node[below=of n2.south](nD){D};
		\node[left=of n4.west](nE){E};
	\end{scope}
	
	\begin{scope}[every node/.style={circle, draw=mygray, dashed}, outer sep=0pt, inner sep=0pt]
		\node[minimum width = 20mm] at (0.26,-0.26) {};
		\node[fit=(n2) (nD)] (n2c) {};
		\node[fit=(n4) (nE)] (n4c) {};
	\end{scope}
	
	\begin{scope}[dotted, thick, blue, every node/.style={scale=0.8, minimum size=0pt, sloped}]
		\draw (nA) -- node[pos=0.3, auto=left, mygreen] {\({d}_A \rightarrow\)} node[pos=0.45, auto=right] {\(\leftarrow {y}_2\)} (n2);
		\draw (nB) -- node[pos=0.4, auto=left, mygreen] {\({d}_B \rightarrow\)} node[pos=0.6, auto=right] {\(\leftarrow {y}_3\)} (n3);
		\draw (nC) -- node[pos=0.4, auto=left, mygreen] {\({d}_C \rightarrow\)} node[pos=0.6, auto=right] {\(\leftarrow {y}_4\)} (n4);
		\draw (nD) -- node[pos=0.4, auto=left, mygreen] {\({d}_D \rightarrow\)} node[pos=0.44, auto=right] {\(\leftarrow {y}_4\)} (n4);
		\draw (nE) -- node[pos=0.4, auto=left, mygreen] {\(\leftarrow {d}_E\)} node[pos=0.6, auto=right] {\({y}_3 \rightarrow\)} (n3);
	\end{scope}
	
	\node[draw=white,minimum width = 4.75cm, minimum height=4.75cm] at (1.55,-1.55) {};
\end{tikzpicture}
            \caption{Controllers as part of the agents}
            \label{fig:withAgents}
        \end{subfigure}
        \caption{Different technical realization of the controller systems.}
        \label{fig:example1_2}
    \end{figure}
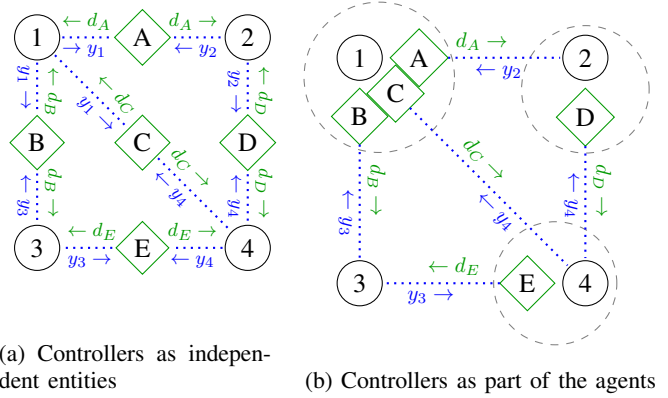
\end{example}
The decision of considering the controllers as independent entities or not may depend on the application. When the agents are likely to leave or rejoin the network often and there exists an independent computing infrastructure, a configuration with controllers as independent entities may be favorable. Otherwise, a configuration with the controllers as part of the agents may be advantageous.



%

\section{Local Design Requirements for Global Optimality and Convergence} \label{sec:distributed_opt}

In this section, we derive local design requirements for the agent systems~\eqref{eq:agents}, the controller systems~\eqref{eq:controllers}, and the communication structure~\eqref{eq:interconnection} such that the networked systems (see Figure~\ref{fig:burger_structure}) achieve global optimality and convergence. To begin with, we revisit the classic optimality conditions for problem~\eqref{eq:optimization_problem}, which the agents and controllers have to jointly fulfill. 

\subsection{Global Optimality Conditions} \label{sec:optimality_conditions}


For the optimality conditions, we distinguish between the unconstrained optimization problem, i.e., \eqref{eq:optimization_problem} without the constraints~\eqref{eq:inequality_constraints}-\eqref{eq:equality_constraints} for any agent $i \in \setOfAgents$, and its constrained version as in \eqref{eq:optimization_problem}. 

In the unconstrained case, we have the following conditions.
\begin{proposition}[\cite{boyd2004convex}]
    Assume the optimization problem~\eqref{eq:optimization_problem} is convex and there are no constraints. Then, $\opt{\yrm}$ is an optimizer if and only if
    \begin{align} \label{eq:optimality_condition}
        \sum_{i \in \setOfAgents} \nabla f_i(\opt{\yrm}) = 0
    \end{align}
    holds. If~\eqref{eq:optimization_problem} is strictly convex, $\opt{\yrm}$ is unique.
\end{proposition}

In the constrained case, the classic optimality conditions are known as the Karush-Kuhn-Tucker (KKT) conditions. Since Slater's condition is assumed to hold, the KKT conditions are necessary and sufficient~\cite[p.~244]{boyd2004convex}. Further, we define the Lagrange function of agent $i\in \setOfAgents$ as
\begin{align} \label{eq:lagrange_function}
    L_i = f_i(\yrm) + \sum_{l \in \setOfIneq{i}} \lambda_{il} g_{il}(\yrm) + \sum_{j \in \setOfEq{i}} \mu_{ij} h_{ij}(\yrm),
\end{align}
with $\lambda_{il}, \mu_{ij}\in \R$, and the global Lagrange function of optimization problem~\eqref{eq:optimization_problem} as $L = \sum_{i\in \setOfAgents} L_i$. For convenience, denote $\lambda_i = \col{\lambda_{il}} \in \R^{\numOfIneq{i}}$ as the multipliers of agent $i$, and analogously for $\mu_i \in \R^{\numOfEq{i}}$.
\begin{proposition}[{\cite[p.~244]{boyd2004convex}}]
    Assume the optimization problem~\eqref{eq:optimization_problem} is convex and let Slater's condition hold. Then, $\opt{\yrm}$ is an optimizer if and only if there exist multipliers $\lambda^*_i,\mu^*_i$ such that the conditions
    \begin{subequations} \label{eq:KKT}
        \begin{align}
            \sum_{i \in \setOfAgents} \nabla_{\yrm} \, L_i(\opt{\yrm},\lambda^*_i,\mu^*_i) &= 0 \\
            g_{il}(\opt{\yrm}) & \leq 0 \label{eq:condition_constraint_1}\\
            h_{ij}(\opt{\yrm}) & = 0 \\
            \lambda_{il}^* & \geq 0 \\
            \lambda_{il}^* \,g_{il}(\opt{\yrm}) & = 0 \label{eq:condition_constraint_4}
        \end{align}
    \end{subequations}
    are fulfilled $\forall i \in \setOfAgents$, $\forall j \in \setOfEq{i}$ and $\forall l \in \setOfIneq{i}$. Due to strict convexity of~\eqref{eq:optimization_problem}, the optimizer $\opt{\yrm}$ is unique. 
\end{proposition} 
Note that both conditions~\eqref{eq:optimality_condition}~and~\eqref{eq:KKT} cannot be solved independently by the agents and controllers, i.e., agents and controllers need to coordinate in order to jointly fulfill these network-wide optimality conditions.



\subsection{Local Design Requirements for Global Optimality} \label{sec:requirements_optimality}

In this subsection, we propose design requirements for the agent systems~\eqref{eq:agents}, the controller systems~\eqref{eq:controllers}, and the communication structure~\eqref{eq:interconnection} such that the closed-loop, networked system in Figure~\ref{fig:burger_structure} meets the network-wide optimality conditions~\eqref{eq:optimality_condition}~or~\eqref{eq:KKT} revisited above in steady state.  



\subsubsection{Unconstrained Optimization}
We start with the requirements in the case of an unconstrained optimization problem, i.e.,~\eqref{eq:optimization_problem} without~\eqref{eq:equality_constraints}-\eqref{eq:inequality_constraints}. 

\begin{theorem} \label{th:requirements_equilibria}
    Consider systems~\eqref{eq:agents} and~\eqref{eq:controllers} interconnected by a connected communication graph~\eqref{eq:interconnection}. If the equilibria of the agent systems~\eqref{eq:agents}
    fulfill
    \begin{align} \label{eq:condition_equilibrium_agents}
        \rl{u}_i = \nabla f_i(\rl{y}_i)
    \end{align}
    and the controllers~\eqref{eq:controllers} are systems with integral action as in Definition~\ref{def:integral_action}, all equilibria $(\rl{y}, \rl{d})$ of the networked system are in the manifold given by $\mathcal{M}_\text{eq} = \{ (\rl{y},\rl{d}) \in \R^n \times \R^n \mid \rl{y} = 1_\numOfAgents \otimes \opt{\yrm} \}$, where $\opt{\yrm}$ is the minimizer of problem~\eqref{eq:optimization_problem}.  
\end{theorem}

\begin{proof}
    With Definition~\ref{def:integral_action}, it follows for the $\Pi_k$ systems that $\rl{\zeta}_k = 0$. With~\eqref{eq:interconnection1} we have 
    \begin{align} \label{eq:consensus}
        0 = \left( \inci \otimes I_n \right)^\trans \rl{y},
    \end{align}
    which is fulfilled if and only if $\rl{y}$ is contained in the kernel of $\left( \inci \otimes I_n \right)^\trans$. Using Proposition~\ref{prop:kernel_kronecker} from Appendix~\ref{sec:appendix_proofs}, it follows from~\eqref{eq:consensus} that all agents are in consensus, i.e., $\rl{y} = 1_\numOfAgents \otimes a_n$, $a_n \in \R^n$. Left-multiplying~\eqref{eq:interconnection2} with the matrix $(1_\numOfAgents^\trans \otimes I_n)$ and taking into account that $1_\numOfAgents^\trans E = 0$ and property P1 of the Kronecker product as in Appendix~\ref{sec:appendix_kronecker}, we have 
    \begin{align} \label{eq:sum_u1}
        (1_\numOfAgents^\trans \otimes I_n) \rl{u} &= - (1_\numOfAgents^\trans \otimes I_n) \left( \inci \otimes I_n \right) \rl{d} \nonumber \\
        & = - (1_\numOfAgents^\trans E \otimes I_n)\rl{d} = 0 .
    \end{align}
    Note that the left-hand side of \eqref{eq:sum_u1} can be rewritten as
    \begin{align} \label{eq:sum_u2}
        (1_n^\trans \otimes I_n) \rl{u} = \sum_{i \in \setOfAgents} \rl{u}_i.
    \end{align}
    Inserting~\eqref{eq:condition_equilibrium_agents} for the equilibrium of the agent systems into~\eqref{eq:sum_u2}, we obtain the optimality condition~\eqref{eq:optimality_condition}. Thus, together with the consensus established above, it holds that $\rl{y}_i = \opt{\yrm}$ $\forall i \in \setOfAgents$ in steady state. Since the controller outputs $d_k$ remain unspecified, any value $\rl{d} \in \R^{n\numOfControllers}$ may be an equilibrium. Thus, the equilibria are described by the manifold $\eqM$ as stated in the theorem.
\end{proof} 

\begin{remark}
    Note that $\eqM$ is a positively invariant set w.r.t.\ the networked system composed of systems~\eqref{eq:agents} and~\eqref{eq:controllers} interconnected by~\eqref{eq:interconnection}. 
\end{remark}

\begin{remark}
    Note that the equilibria of systems~\eqref{eq:agents} and~\eqref{eq:controllers} interconnected by~\eqref{eq:interconnection} are not unique, since $\rl{z}$ remains unspecified apart from $\rl{\zeta}_k = 0$. Depending on the initial values of the states $z_k$ and $x_i$, different equilibria arise in closed-loop operation. However, for all equilibria it holds that $\rl{y}_i = \opt{\yrm}$ $\forall i \in \setOfAgents$. 
\end{remark}

\begin{remark}
    Even if not included in the definition of the manifold $\eqM$, the equilibria for the controller inputs $\zeta_k$ are also uniquely determined by~\eqref{eq:condition_equilibrium_agents}, viz. $\zeta_k = 0$, as it becomes clear from the proof of Theorem~\ref{th:requirements_equilibria}. 
\end{remark}

\subsubsection{{Constrained Optimization}}
In the following, we state the design requirements for the agent and controller systems in the case of a constrained optimization problem as in~\eqref{eq:optimization_problem}. 

\begin{theorem} \label{th:requirements_equilibria_constrained}
    Consider systems~\eqref{eq:agents} and~\eqref{eq:controllers} interconnected by a connected communication graph~\eqref{eq:interconnection}. If the equilibria of the agent systems~\eqref{eq:agents}
    fulfill
    \begin{subequations} \label{eq:condition_equilibrium_agents_constrained}
        \begin{align}
            \rl{u}_i & = \nabla_{\yrm} \, L_i(\rl{y}_i, \rl{\lambda}_i, \rl{\mu}_i) \label{eq:condition_equilibrium_agents_constrained_1} \\
            g_{il}(\rl{y}_i) & \leq 0 \label{eq:condition_equilibrium_agents_constrained_2}   \\
            h_{ij}(\rl{y}_i) & = 0 \label{eq:condition_equilibrium_agents_constrained_3} \\
            \rl{\lambda}_{il}^* & \geq 0 \label{eq:condition_equilibrium_agents_constrained_4} \\
            \rl{\lambda}_{il}^*g_{il}(\rl{y}_i) & = 0 \label{eq:condition_equilibrium_agents_constrained_5}
        \end{align}
    \end{subequations}
    $\forall l \in \setOfIneq{i}$ and $\forall j \in \setOfEq{i}$, and the controllers~\eqref{eq:controllers} are systems with integral action as in Definition~\ref{def:integral_action}, all equilibria $(\rl{y}, \rl{d})$ of the networked system are in the manifold given by $\mathcal{M}_\text{eq} = \{ (\rl{y},\rl{d}) \in \R^n \times \R^n \mid \rl{y} = 1_\numOfAgents \otimes \opt{\yrm} \}$, where $\opt{\yrm}$ is the minimizer of the constrained problem~\eqref{eq:optimization_problem}.  
\end{theorem}
\begin{proof}
    This proof largely follows the proof of Theorem~\ref{th:requirements_equilibria} for the unconstrained case. Due to the interconnection~\eqref{eq:interconnection}, we obtain consensus on the outputs $y_i$ of the agents using Proposition~\ref{prop:kernel_kronecker} of Appendix~\ref{sec:appendix_proofs} and the integral action of the controller systems. Further, following the same procedure as in the proof of Theorem~\ref{th:requirements_equilibria}, we obtain~\eqref{eq:sum_u1}. Inserting condition~\eqref{eq:condition_equilibrium_agents_constrained_1} on the left side and taking into account that all outputs $\rl{y}_i$ are in consensus together with~\eqref{eq:condition_equilibrium_agents_constrained_2}-\eqref{eq:condition_equilibrium_agents_constrained_5}, we obtain the KKT optimality condition~\eqref{eq:KKT}.  
\end{proof}


Due to the proposed algorithm structure in Figure~\ref{fig:burger_structure}, the classic optimality conditions~\eqref{eq:optimality_condition} and~\eqref{eq:KKT} hold per design at steady state. Consequently, in contrast to the literature, no global initialization or other conditions have to be imposed to achieve global optimality. This implies that agents with optimization dynamics complying with the requirements given in~\eqref{eq:condition_equilibrium_agents} or~\eqref{eq:condition_equilibrium_agents_constrained} may leave or rejoin the networked optimization, while the remaining agents in the network automatically have an equilibrium at the new global optimum $\opt{\yrm}$ arising from the new network configuration. 

\begin{remark}\label{rem:nullspace}
	Note that in the proofs of Theorems~\ref{th:requirements_equilibria}~and~\ref{th:requirements_equilibria_constrained}, we only used the fact that the nullspace of the transposed communication incidence matrix $\inci^\trans$ is of dimension one and contains the vector $1_\numOfAgents$ to achieve consensus and global optimality. Thus, any matrix fulfilling this nullspace property can be used, which is studied in the next subsection. 
\end{remark}


%
%
\subsubsection{Communication Topology} \label{sec:requirements_communication}
At this point, we have posed design requirements on the agent and controller dynamics such that all equilibria correspond to the global optimizer $\opt{\yrm}$.
In this subsection, we study the effects of different agent communication structures on global optimality. First, we consider the case where a controller is connected to more than two agents. We show that such generalized communication structures exhibit a certain skew-symmetry property, and are thus amenable to common passivity theory. Afterwards, we extend our framework to allowing directed communication topologies, in which a controller does not have to send its state back to the agents from whom it received the estimations, but can instead send it to other agents in the network. Hence, we refer to such directed communication topologies as non-symmetric communication structures.

\paragraph{\textbf{Generalized Undirected Topology}}

Consider the interconnection equations~\eqref{eq:interconnection} in matrix form
\begin{align} \label{eq:communication_bipartite_incidence}
    \mat{-u \\ \zeta} = \mat{0 & \inci \otimes I_n \\ \left(\inci \otimes I_n\right)^\trans & 0} \mat{y \\ d}.
\end{align}
%
We observe that the matrix in~\eqref{eq:communication_bipartite_incidence} is symmetric, which implies a bidirectional or undirected communication topology, i.e., a controller that receives an output of agent $i \in \setOfAgents$ sends back its state to that agent $i$. This symmetry property is instrumental in the derivation of the Lyapunov function in Section~\ref{sec:convergence_undirected}. However, the only property of the communication incidence matrix $\inci^\trans$ we used in Theorems~\ref{th:requirements_equilibria}~and~\ref{th:requirements_equilibria_constrained} is that $\inci^\trans$ has a nullspace of dimension one spanned by the $1_\numOfAgents$ vector, i.e.,
\begin{equation}\label{eq:nullspace_property}
    \ker(\inci^\trans) = \left\{ \alpha 1_\numOfAgents \mid \alpha \in \R \right\}.
\end{equation}
(recall Remark~\ref{rem:nullspace}). The incidence matrices of connected graphs are, however, only a subset of the matrices with this nullspace property~\eqref{eq:nullspace_property}. Thus, we can readily use any matrix $R$ with the nullspace property~\eqref{eq:nullspace_property}, which is the only requirement for the communication structure. In particular, all the results of Theorems~\ref{th:requirements_equilibria}~and~\ref{th:requirements_equilibria_constrained} hold for such an $R$. Figure~\ref{fig:bipartite_symmetric} shows the effect of generalized communication for the graph structure. Observe that a controller is connected to more than two nodes, instead of being associated with an edge between only two agents as before. With the generalized communication structure, the second equation in~\eqref{eq:communication_bipartite_incidence} computes a weighted average of the outputs $y_i$. Such a communication matrix is called mixing matrix in~\cite{shi2015extra}. 
\begin{figure}[t]
    \begin{subfigure}[b]{0.2\textwidth}
        \centering
        \begin{align*}
            R = \mat{1 & 0 \\ -2 & 1 \\ 1 & -1}
        \end{align*}
        \caption{Communication matrix}
        \label{fig:bipartite_symmetric_adjacency_matrix}
    \end{subfigure}
    \hfill
    \begin{subfigure}[b]{0.27\textwidth}
        \centering
        \begin{tikzpicture}
	\begin{scope}[node distance={1.25cm and 1.25cm}, every node/.style={draw=black}, minimum size=6mm, outer sep=1pt, inner sep=2pt, on grid]
		\node[circle] (n1) at (0,0) {1};
		\node[circle, right=of n1, right=1.75](n2){2};
		\node[diamond, draw=mygreen, below=of n1](nA){A};
		\node[diamond, draw=mygreen, below=of n2](nB){B};
		\node[circle, below=0.75 of {$(nA)!0.5!(nB)$}](n3){3};
	\end{scope}
	
	\begin{scope}[dotted, thick, blue]
		\draw (n1) -- node[left] {\small 1} (nA);
		\draw (nA) -- node[pos=0.5, auto=right] {\small 1}  (n3);
		\draw (n2) -- node[pos=0.5, auto=right] {\small -2}  (nA.east);
		\draw (n2) -- node[pos=0.5, auto=left] {\small 1}  (nB);
		\draw (n3) -- node[pos=0.5, auto=right] {\small -1}  (nB);
	\end{scope}
\end{tikzpicture}
        \caption{Graph}
        \label{fig:bipartite_symmetric_graph}
    \end{subfigure}
    \caption{Network composed of 3 agents and 2 controllers with a generalized symmetric communication structure.}
        \label{fig:bipartite_symmetric}
\end{figure}
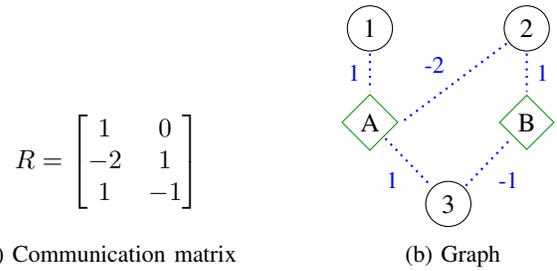
The next proposition provides a relation between the number of agents and controllers necessary to obtain the nullspace property and thus globally optimal equilibria. 
\begin{proposition}
    The network has to contain at least $\numOfAgents - 1$ controllers in order to fulfill the nullspace property. 
\end{proposition}
\begin{proof}
    The nullspace property states that the communication matrix $R^\trans \in \R^{\numOfControllers\times\numOfAgents}$ should have a nullspace of dimension one spanned by the vector $1_\numOfAgents$. Assume $\numOfControllers < \numOfAgents-1$. Then, the rank of $R^\trans \in \R^{\numOfControllers\times \numOfAgents}$ is at most $\numOfAgents-2$, since it has at most $\numOfAgents-2$ rows. Due to the rank-nullity theorem~\cite{horn2012matrix}, having $\numOfAgents$ columns and a rank of $N-2$ implies that the kernel of $R^\trans$ must have dimension two, which violates the nullspace property.
\end{proof}

\paragraph{\textbf{Directed Graphs and Non-Symmetric Communication Structure}}

In this subsection, we explore the requirements on the communication topology when using a directed communication structure. In such a case, the information exchange may not be bidirectional between agents and controllers. This is described by a non-symmetric communication structure, composed of two different matrices $R_\setOfAgents$ and $R_\setOfControllers \in \R^{\numOfAgents\times \numOfControllers}$ of identical size, i.e.,
\begin{align} \label{eq:interconnection_nonsymmetric}
    \mat{-u \\ \zeta} = \mat{0 & R_\setOfControllers \otimes I_n \\ \left(R_\setOfAgents \otimes I_n\right)^\trans & 0} \mat{y \\ d}.
\end{align}
The matrix $R_\setOfAgents$ describes the information flow from agents to controllers, $R_\setOfControllers$ from controllers to agents. An example is given in Figure~\ref{fig:bipartite_nonsymmetric}. The directed communication from agents to controllers described by $R_\setOfAgents$ is displayed in blue, and the directed communication from controllers to agents $R_\setOfControllers$ in red. Observe that an agent may send its output $y_i$ to a controller which does not send its integrator state back. As before, the only requirement for global optimality in Theorems~\ref{th:requirements_equilibria}~and~\ref{th:requirements_equilibria_constrained} is that the communication matrices $R_\setOfAgents$ and $R_\setOfControllers$ have, individually, the stated nullspace property~\eqref{eq:nullspace_property}. Graphically, the nullspace property of $R_\setOfAgents$ means that the weights of incoming edges at any controller have to sum up to zero. Similarly, the nullspace property of $R_\setOfControllers$ means that the weights of the outgoing edges at the controller have to sum up to zero (see~Figure~\ref{fig:bipartite_nonsymmetric}).   

The main issue with the directed communication topology is that it does not generally translate into a skew-symmetric interconnection. This is, however, crucial for a direct passivity-based stability analysis. The matrix in~\eqref{eq:interconnection_nonsymmetric} describes a non-symmetric adjacency matrix of a bipartite graph. In order to cope with that, we will need a greater excess of \gls{eip} on the agent dynamics, as is explained in Section~\ref{sec:convergence_directed} in detail. 
\begin{figure}[t]
    \begin{subfigure}[b]{0.2\textwidth}
        \centering
        \begin{align*}
            R_\setOfAgents = \mat{1 & 0 \\ -2 & 1 \\ 1 & -1} \\ \\
            R_\setOfControllers = \mat{0 & 1 \\ 2 & 0 \\ -2 & -1} 
        \end{align*}
        \caption{Communication matrices}
        \label{fig:bipartite_nonsymmetric_adjacency_matrix}
    \end{subfigure}
    \hfill
    \begin{subfigure}[b]{0.27\textwidth}
        \centering
        \begin{tikzpicture}
	\begin{scope}[node distance={1.75cm and 1.75cm}, every node/.style={draw=black}, minimum size=6mm, outer sep=1pt, inner sep=2pt, on grid]
		\node[circle] (n1) at (0,0) {1};
		\node[circle, right=of n1, right=1.9](n2){2};
		\node[diamond, draw=mygreen, below=of n1](nA){A};
		\node[diamond, draw=mygreen, below=of n2](nB){B};
		\node[circle, below=1 of {$(nA)!0.5!(nB)$}](n3){3};
	\end{scope}
	
	\begin{scope}[dotted, thick, blue,  bend left=15, -latex, every node/.style={scale=0.8, pos=0.8, minimum size=0pt}]
		\draw (n1) to[bend right=15] node[pos=0.5, auto=left] {\small 1} (nA);
		\draw (n3) to node[pos=0.5, auto=left] {\small 1}  (nA);
		\draw (n2) to node[pos=0.25, auto=left] {\small -2}  (nA.east);
		\draw (n2) to node[pos=0.5, auto=left] {\small 1}  (nB);
		\draw (n3) to node[pos=0.5, auto=left] {\small -1}  (nB);
	\end{scope}
	
	\begin{scope}[dotted, thick, red,  bend left=15, -latex, every node/.style={scale=0.8, pos=0.8, minimum size=0pt, auto=left}]
		\draw (nA) to node[pos=0.8, auto=left] {\small 2}  (n2);
		\draw (nA) to node[pos=0.5, auto=left] {\small -2}  (n3);
		\draw (nB) to node[pos=0.2, auto=left] {\small 1}  (n1);
		\draw (nB) to node[pos=0.5, auto=left] {\small -1}  (n3);
	\end{scope}
	
\end{tikzpicture}
        \caption{Graph}
        \label{fig:bipartite_nonsymmetric_graph}
    \end{subfigure}
    \caption{Network composed of 3 agents and 2 controllers with a directed, non-symmetric communication structure.}
        \label{fig:bipartite_nonsymmetric}
\end{figure}
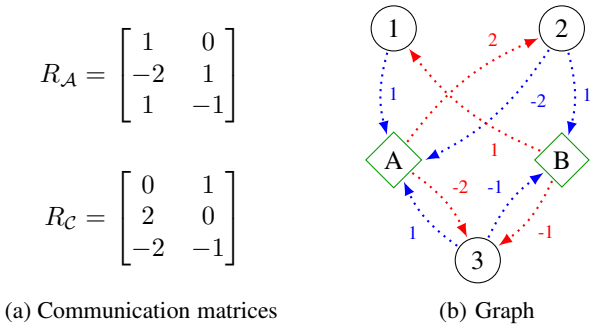
\begin{remark}
    Note that in both undirected and directed communication structures, the nullspace property can be trivially fulfilled individually by the controllers and does not require any global coordination. Provided that all agents send their output $y_i$, each controller can assign individually a weight to each incoming (or, similarly, outgoing) edge such that the sum is zero. 
\end{remark}

In the next subsection, we state requirements for the agent and controller dynamics to converge to the global optimizers characterized in this subsection. 

\subsection{Local Design Requirements for Convergence with Undirected Communication Topologies} \label{sec:requirements_convergence} 


The design requirements we pose on the agent and controller systems to ensure convergence to the global optimizers characterized in Section~\ref{sec:requirements_optimality} are, in a broad sense, \gls{eip} properties~\cite{simpson2018equilibrium,arcak2016networks}. Depending on the scenario, i.e., unconstrained or constrained optimization with undirected or directed communication topologies, different \gls{eip} properties are necessary. The following theorem provides local design requirements for achieving convergence in the case of undirected communication topologies.


\begin{theorem} \label{th:requirements_convergence}
    Let the agents $i \in \setOfAgents$ and controller systems $k \in \setOfControllers$ fulfill the design requirements of Theorems~\ref{th:requirements_equilibria} or~\ref{th:requirements_equilibria_constrained} and Definition~\ref{def:integral_action}, respectively. Further, let the agents be interconnected by a communication structure~\eqref{eq:communication_bipartite_incidence} fulfilling property~\eqref{eq:nullspace_property}. In addition, let the agent systems $i \in \setOfAgents$ be \gls{eio} and \gls{oseip}$(\ofpIdx_i)$ with $\ofpIdx_i>0$ and positive definite storage function $S_i$, and the controller systems $k \in \setOfControllers$ be \gls{eio} and \gls{eip} with positive definite storage function $W_k$. Then, all trajectories of the networked system converge to a point in the invariant equilibrium manifold $\eqM$ as specified in Theorem~\ref{th:requirements_equilibria} or~\ref{th:requirements_equilibria_constrained}, which implies that the estimation $y_i$ of all agents converges to the global optimizer $\opt{\yrm}$.   
\end{theorem}
\begin{proof}
    Consider the network-wide Lyapunov function
    \begin{align} \label{eq:lyapunov}
        V = \sum_{i \in \setOfAgents} S_i + \sum_{k \in \setOfControllers} W_k 
    \end{align}
    and the time derivative using the \gls{oseip}$(\ofpIdx_i)$ and \gls{eip} properties
    \begin{align}
        \dot{V} & \leq \sum_{i \in \setOfAgents} \left( \err{u}_i^\trans \err{y}_i - \ofpIdx_i \err{y}_i^\trans \err{y}_i \right) + \sum_{k \in \setOfControllers}  \err{\zeta}_k^\trans \err{d}_k . \nonumber
    \end{align}
    Written in stacked network variables, we have
    \begin{align}
        \dot{V} &\leq \err{u}^\trans \err{y} - \err{y}^\trans \left(\diag{\ofpIdx}\otimes I_n \right) \err{y} + \err{\zeta}^\trans \err{d}   \nonumber \\
        & = - \err{y}^\trans K \err{y} ,
    \end{align}
    where the last equality follows by inserting the interconnection topology~\eqref{eq:communication_bipartite_incidence} into the first and third term, and introducing $K = \diag{\ofpIdx}\otimes I_n$. The matrix $K$ is positive definite, since $\ofpIdx_i >0$. Invoking Lasalle's Invariance Theorem, all trajectories converge to $\err{y} = 0$, which corresponds to the manifold $\eqM$, and $y_i = \opt{\yrm}$. Since all agents are in consensus, $\zeta_k$ is zero and with \gls{eio}, all $z_k$, $k\in\setOfControllers$ and $x_i$, $i\in\setOfAgents$ are constant.  
\end{proof}
%
Note that Theorem~\ref{th:requirements_convergence} covers both unconstrained and constrained optimization. The state-of-the-art algorithms for the constrained case in continuous time, however, comprise discontinuous agent dynamics~\cite{hatanaka2018passivity,cherukuri2016asymptotic} or differential inclusions~\cite{liu_SecondOrderMultiAgent_2015,zeng_DistributedContinuousTime_2017}. In this case, Theorem~\ref{th:requirements_convergence} can still be applied to such dynamics using generalized gradients and Lyapunov theory for discontinuous dynamical systems (e.g.~\cite[Th.~1]{cortes2008discontinuous} or~\cite[Lemma~1]{bacciotti2006nonpathological}). To do so, the following additional assumption needs to hold (see Section~\ref{sec:constrained_optimization} for details):
\begin{assumption}  \label{assum:existence_solution}
	Consider a discontinuous dynamical system composed of agents $i \in \setOfAgents$ and controllers $k\in\setOfControllers$ interconnected by a communication structure~\eqref{eq:communication_bipartite_incidence}. For all admissible initial states, a solution of the discontinuous dynamical system exists. 
\end{assumption}
\begin{remark}
	We state Assumption~\ref{assum:existence_solution} as an additional requirement instead of merging it into Theorem~\ref{th:requirements_convergence}, since technically, the constrained case does not necessitate discontinuous agent dynamics. Continuous agent dynamics for a constrained, distributed optimization problem as in \eqref{eq:optimization_problem} may be conceivable, but, to the best of the authors' knowledge, have not been investigated yet. 
\end{remark}

In addition, note that the global optimality design requirements in Theorems~\ref{th:requirements_equilibria}~and~\ref{th:requirements_equilibria_constrained} are local properties for each agent and controller subsystem. Also, \gls{oseip} and \gls{eip} are per design local subsystem properties. Therefore, with the proposed framework, there is no need for all agents to follow a specific, given algorithm. Networks composed of agents with heterogeneous optimization dynamics converge to the global optimizer, provided that the local design requirements for global optimality and convergence given in Theorems~\ref{th:requirements_equilibria}, \ref{th:requirements_equilibria_constrained}, and \ref{th:requirements_convergence}, as well as the nullspace requirement \eqref{eq:nullspace_property} for the communication structure are satisfied. 
This makes our distributed optimization framework easily scalable and enables agents to leave or (re)join the optimization without compromising global optimality and convergence. 
%
%
\begin{remark}
	The requirement of \gls{oseip} agent dynamics may be restrictive in some cases and can be relaxed for both the unconstrained and constrained problem with an undirected communication topology under minor additional information on the agent dynamics (see Section~\ref{sec:convergence_undirected}). 
\end{remark}
%


\section{Agent and Controller Dynamics} \label{sec:agent_controller_dynamics}

In this section, we propose different agent and controller dynamics and analyze their \gls{eip} properties. Different \gls{eip} properties are necessary for satisfying the requirements established in Section~\ref{sec:distributed_opt} and ensuring global optimality and convergence in different scenarios (i.e. the analyses in Sections~\ref{sec:convergence_undirected} and~\ref{sec:convergence_directed}). 

\subsection{Agent Dynamics for Unconstrained Optimization}

Consider the system dynamics for the agents 
\begin{align} \label{eq:agents_concrete}
    \Sigma_i: 
    \begin{array}{cl}
        \dot{x}_i & = -\alpha_i \nabla f_i(x_i) + \alpha_i u_i \\
        y_i & = x_i,
    \end{array}
\end{align}
where $\alpha_i > 0$ is a design parameter. Condition~\eqref{eq:condition_equilibrium_agents} of Theorem~\ref{th:requirements_equilibria} is trivially fulfilled. The EIP properties of~\eqref{eq:agents_concrete} are studied in the next proposition.

\begin{proposition} \label{prop:agents_oseip}
    Consider the storage function $S_i(\err{x}_i) = \frac{1}{2\alpha_i}\err{x}^\top_i \err{x}_i$. System~\eqref{eq:agents_concrete} is \gls{eip}, if $f_i$ is convex, and \gls{oseip}$(\convIdx_i)$, if $f_i$ is $\convIdx_i$-strongly convex.  
\end{proposition}

\begin{proof}
    Consider the time derivative of the storage function 
    \begin{align}
        \dot{S}_i &= \frac{1}{\alpha_i}\err{x}_i^\trans \dot{x}_i \\
        & = \frac{1}{\alpha_i} \err{x}_i^\trans (-\alpha_i \nabla f_i(x_i) + \alpha_i u_i). \label{eq:storage_function_derivative1}
    \end{align}
    Next, add $\err{x}_i^\trans \left( \nabla f_i(\rl{x}_i) - \rl{u}_i \right)$, which equals to zero for any equilibrium $\rl{x}_i$, $\rl{u}_i$, to~\eqref{eq:storage_function_derivative1} to obtain
    \begin{align}
        \dot{S}_i &=  \err{x}_i^\trans (- \nabla f_i(x_i) +  u_i) - \err{x}_i^\trans (- \nabla f_i(\rl{x}_i) +  \rl{u}_i) \nonumber \\
        & = \err{y}_i^\trans \err{u}_i - \err{y}_i^\trans (\nabla f_i(y_i) - \nabla f_i(\rl{y}_i)). \label{eq:storage_function_derivative2}
    \end{align}
    Taking into account the convexity condition for differentiable functions~\eqref{eq:convexity_differentiable}, we obtain from~\eqref{eq:storage_function_derivative2} $\dot{S}_i \leq \err{y}_i^\trans \err{u}_i$, if $f_i$ is convex, which shows \gls{eip} of~\eqref{eq:agents_concrete}. Furthermore, we obtain from~\eqref{eq:storage_function_derivative2} $\dot{S}_i \leq \err{y}_i^\trans \err{u}_i - \convIdx_i \err{y}_i^\trans \err{y}_i$ if $f_i: \R^n \rightarrow \R$ is $\convIdx_i$-strongly convex, which proves \gls{oseip}$(\convIdx_i)$ of~\eqref{eq:agents_concrete}.   
\end{proof}

\begin{corollary} \label{cor:agents_eid}
    The system dynamics of the agents~\eqref{eq:agents_concrete} is also \gls{eip} with supply rate $w(\err{u}_i, \err{y}_i) = \err{y}_i^\trans \err{u}_i - \Psi_i(\err{y}_i)$ with $\Psi_i(\err{y}_i) =  \err{y}_i^\trans (\nabla f_i(y_i) - \nabla f_i(\rl{y}_i))$, as follows directly from~\eqref{eq:storage_function_derivative2}. Note that $\Psi_i(\err{y}_i) \geq 0$ if the function $f_i:\R^n \rightarrow \R$ is convex. 
\end{corollary}

In some cases, agent dynamics having additionally an excess of input passivity may be necessary, e.g., when using directed communication topologies (c.f.\ Section~\ref{sec:convergence_directed}). Therefore, we propose the modified agent dynamics
\begin{align} \label{eq:agents_durchgriff}
    \Sigma_i: 
    \begin{array}{cl}
        \dot{x}_i & = -\alpha_i \nabla f_i(x_i + \feedt_i u_i) + \alpha_i u_i \\
        y_i & = x_i + \feedt_i u_i,
    \end{array}
\end{align}
with $\feedt_i \in \R_{\geq 0}$. Note that the dynamics~\eqref{eq:agents_durchgriff} fulfill the equilibrium requirement of Theorem~\ref{th:requirements_equilibria}. This particular agent dynamics exhibits a direct feedthrough from input to output, which is expected to induce an excess of input-feedforward \gls{eip}, as shown analytically for a single integrator in~\cite{sepulchre2012constructive}. This is formally proven in the next proposition.  

\begin{proposition} \label{prop:agents_oseip_iseip}
    Consider the storage function $S_i(\err{x}_i) = \frac{1}{2 \alpha_i} \err{x}_i^\trans \err{x}_i$. System~\eqref{eq:agents_durchgriff} is \isoseip$(\ifpIdx_i,\ofpIdx_i)$ w.r.t.\ the supply rate $w(\err{u}_i,\err{y}_i) = \err{y}_i^\trans \err{u}_i - \ifpIdx_i \err{u}_i^\trans \err{u}_i - \ofpIdx_i \err{y}_i^\trans \err{y}_i$, with $\ifpIdx_i=\frac{1}{2} \feedt_i$ and $\ofpIdx_i=\convIdx_i - \frac{1}{2} \feedt_i \lipIdx_i$, if the objective function $f_i:\R^n \rightarrow \R$ is $\convIdx_i$-strongly convex and has $\lipIdx_i$-Lipschitz gradients. 
\end{proposition}

\begin{proof}
    Consider the derivative of the storage function 
    \begin{align}
        \dot{S}_i &= \frac{1}{\alpha_i} \err{x}_i^\trans \dot{x}_i \nonumber \\
        & = \frac{1}{\alpha_i} \err{x}_i^\trans \left( -\alpha_i \nabla f_i(x_i + \feedt_i u_i) + \alpha_i u_i  \right) \nonumber \\
        & = \err{x}_i^\trans u_i - \err{x}_i^\trans \nabla f_i(x_i + \feedt_i u_i) \label{eq:ifp_storage_function_derivative}
    \end{align}
    Next, add $\err{x}_i^\trans \left( \nabla f_i(\rl{x}_i + \feedt_i \rl{u}_i) - \rl{u}_i  \right)$, which equals to zero, to~\eqref{eq:ifp_storage_function_derivative} and obtain
    \begin{align}
        \dot{S}_i &= \err{x}_i^\trans \err{u}_i - \err{x}_i^\trans \left(\nabla f_i(x_i + \feedt_i u_i) - \nabla f_i(\rl{x}_i+ \feedt_i \rl{u}_i)\right). \label{eq:ifp_storage_function_derivative2}
    \end{align}
    Substituting $y_i = x_i + \feedt_i u_i$ and $\err{x}_i = \err{y}_i - \feedt_i \err{u}_i$ in~\eqref{eq:ifp_storage_function_derivative2} and defining for convenience $\err{f}_i(y_i,\rl{y}_i) := f_i(y_i) - f_i(\rl{y}_i)$ we have
    \begin{align}
        \dot{S}_i &= \left(\err{y}_i-\feedt_i \err{u}_i\right)^\trans \err{u}_i - \left(\err{y}_i - \feedt_i \err{u}_i\right)^\trans \nabla \err{f}_i(y_i,\rl{y}_i) \nonumber \\
        & = \err{y}_i^\trans \err{u}_i - \feedt_i \err{u}_i^\trans \err{u}_i - \err{y}_i^\trans \nabla \err{f}_i(y_i,\rl{y}_i)  + \feedt_i \err{u}_i^\trans \nabla \err{f}_i(y_i,\rl{y}_i).  \label{eq:ifp_storage_function_derivative3}
    \end{align}
    In the following, we derive an upper bound for the third and fourth term in~\eqref{eq:ifp_storage_function_derivative3}, since they are not in a form amenable to the quadratic supply rate provided in the statement of the proposition. Taking into account the strong convexity condition~\eqref{eq:convexity_differentiable}, it holds for the third term that
    \begin{align}
        - \err{y}_i^\trans \nabla \err{f}_i(y_i,\rl{y}_i) \leq - \convIdx_i \err{y}_i^\trans \err{y}_i. \label{eq:overapprox_convexity}
    \end{align}
    For an upper bound of the fourth term, first consider that the shifted dynamics $\dot{\err{x}}_i = -\alpha_i \nabla \err{f}_i(y_i,\rl{y}_i) + \alpha_i \err{u}_i$ and note that 
    \begin{align}
        \frac{1}{\alpha_i^2}\dot{\err{x}}_i^\trans \dot{\err{x}}_i &= \left( \err{u}_i - \nabla \err{f}_i(y_i,\rl{y}_i) \right)^\trans \left( \err{u}_i - \nabla \err{f}_i(y_i,\rl{y}_i) \right) \nonumber \\ 
        & = \err{u}_i^\trans \err{u}_i + \nabla \err{f}_i(y_i,\rl{y}_i)^\trans \nabla \err{f}_i(y_i,\rl{y}_i) \nonumber \\ & \quad - 2 \err{u}_i^\trans \nabla \err{f}_i(y_i,\rl{y}_i) \geq 0 \label{eq:ifp_storage_function_derivative4}
    \end{align}
    always holds. Furthermore, taking into account the $\lipIdx_i$-Lipschitz property of $\nabla f_i$, we obtain for \eqref{eq:ifp_storage_function_derivative4} 
    \begin{align}
        \err{u}_i^\trans \nabla \err{f}_i(y_i,\rl{y}_i) & \leq \frac{1}{2}\err{u}_i^\trans \err{u}_i + \frac{1}{2} \nabla \err{f}_i(y_i,\rl{y}_i)^\trans \nabla \err{f}_i(y_i,\rl{y}_i) \nonumber \\
        & \leq \frac{1}{2}\err{u}_i^\trans \err{u}_i + \frac{1}{2} \lipIdx_i \err{y}_i^\trans \err{y}_i. \label{eq:overapprox_lipschitz}
    \end{align}
    Substituting~\eqref{eq:overapprox_convexity}~and~\eqref{eq:overapprox_lipschitz} in the third and fourth term, respectively, in the derivative of the storage function~\eqref{eq:ifp_storage_function_derivative3}, we obtain
    \begin{align}
        \dot{S}_i & \leq \err{y}_i^\trans \err{u}_i - \frac{1}{2} \feedt_i \err{u}_i^\trans \err{u}_i - \left( \convIdx_i - \frac{1}{2} \feedt_i \lipIdx_i \right) \err{y}_i^\trans \err{y}_i, \label{eq:ifp_supply_rate}
    \end{align}
    which concludes the proof.
\end{proof}

In the last proposition, we have shown that the agent dynamics~\eqref{eq:agents_durchgriff} are \isoseip$(\ifpIdx_i,\ofpIdx_i)$ with indices $\ifpIdx_i = \frac{1}{2} \feedt_i$ and $\ofpIdx_i = \convIdx_i - \frac{1}{2} \feedt_i \lipIdx_i$. Note that by choosing $\feedt_i = 0$, we recover exactly the result of Proposition~\ref{prop:agents_oseip}. As known from the literature, the \gls{oseip} and \gls{iseip} indices cannot be chosen independently of each other. An increase of the \gls{oseip} index implies a decrease of the \gls{iseip} index and vice versa~\cite{simpson2018equilibrium}, which also applies to~\eqref{eq:ifp_supply_rate}. However, even if typically the Lipschitz index is greater than the strong convexity parameter, i.e., $\lipIdx_i \geq \convIdx_i$, by choosing $\feedt_i$ small enough, we always recover \isoseip$(\ifpIdx_i,\ofpIdx_i)$ for system~\eqref{eq:agents_durchgriff} with an excess of passivity, i.e., $\ifpIdx_i, \ofpIdx_i>0$. The largest bound for $\feedt_i$ can be computed by imposing that $\ofpIdx_i>0$ and solving for $\feedt_i$, i.e., $\feedt_i < \frac{2\convIdx_i}{\lipIdx_i}\leq 2$. Thus, the upper bounds on the indices are $\ofpIdx_i = \convIdx_i$ and $\ifpIdx_i = 1$. 

For the case when the agents employ a quadratic objective function, we can state a counterpart of Proposition~\ref{prop:agents_oseip_iseip} without requiring strong convexity and without over approximations, yielding sharper bounds for the \gls{oseip} and \gls{iseip} indices. Strong convexity is necessary in Proposition~\ref{prop:agents_oseip_iseip}, since without $\convIdx_i > 0$, $\ofpIdx_i$ cannot be greater or equal to zero.

\begin{corollary}  \label{cor:quadratic}
    Assume an agent $i \in\setOfAgents$ has a convex quadratic objective function of the form $f_i(x_i) = \frac{1}{2}x_i^\trans Q_i x_i + q_i^\trans x_i + c_i$, with $Q_i\in\R^{n\times n}$ positive semidefinite, $q_i \in \R^n$ and $c_i\in\R$. Consider the storage function $S_i(\err{x}_i) = \frac{1}{2\alpha_i}\err{x}_i^\trans G_i \err{x}_i$ with $G_i = \left(I_n + \feedt_i Q_i \right)^{-1}$ positive definite for $\feedt_i\geq0$. System~\eqref{eq:agents_durchgriff} is \isoseip($\ifpIdx_i, \ofpIdx_i$)~with $\ofpIdx_i = \eigmin{G_i Q_i}$ and $\ifpIdx_i = \eigmin{\feedt_i G_i}$. 
\end{corollary}

\begin{proof}
    With~\eqref{eq:agents_durchgriff}, the time derivative of the storage function $S_i$ is given by
    \begin{align}
        \dot{S}_i & = \frac{1}{\alpha_i} \err{x}_i^\trans G_i \dot{x}_i \nonumber \\
        & = \frac{1}{\alpha_i} \err{x}_i^\trans G_i \left( -\alpha_i Q_i (x_i + \feedt_i u_i) - \alpha_i q_i + \alpha_i u_i \right) \nonumber \\
        & = \err{x}_i^\trans G_i \left( - Q_i (x_i + \feedt_i u_i) - q_i + u_i \right). \label{eq:storage_function_quadratic_derivative}
    \end{align}
    Adding $\err{x}_i^\trans G_i \left(Q_i (\rl{x}_i + \feedt_i \rl{u}_i) +q_i - \rl{u}_i\right)$, which equals to zero, to~\eqref{eq:storage_function_quadratic_derivative}, we obtain
    \begin{align}
        \dot{S}_i & = \err{x}_i^\trans G_i\left(- Q_i (\err{x}_i + \feedt_i \err{u}_i)  + \err{u}_i \right)  \nonumber \\
        & = \err{x}_i^\trans G_i (I_n - \feedt_i Q_i)\err{u}_i - \err{x}_i^\trans G_i Q_i \err{x}_i. \label{eq:storage_function_quadratic_derivative_2}
    \end{align}
    By substituting $\err{x}_i = \err{y}_i - \feedt_i \err{u}_i$ in~\eqref{eq:storage_function_quadratic_derivative_2} and expanding all the products, we have
    \begin{align}
        \dot{S}_i & = \err{y}_i^\trans (G_i - \feedt_i G_i Q_i)\err{u}_i - (\err{y}_i - \feedt_i \err{u}_i)^\trans G_i Q_i (\err{y}_i - \feedt_i \err{u}_i) \nonumber \nonumber \\ & \quad - \feedt_i \err{u}_i^\trans G_i (I_n - \feedt_i Q_i)\err{u}_i \nonumber \\ \nonumber 
        & = \err{y}_i^\trans (G_i - \feedt_i G_i Q_i)\err{u}_i + \feedt_i \err{y}_i^\trans Q_i G_i \err{u}_i + \feedt_i \err{y}_i^\trans G_i Q_i \err{u}_i \\ \nonumber  & \quad - \err{y}_i^\trans G_i Q_i \err{y}_i  - \feedt_i^2 \err{u}_i^\trans G_i Q_i \err{u}_i - \feedt_i \err{u}_i^\trans G_i (I_n - \feedt_i Q_i)\err{u}_i \\ \nonumber 
        & = \err{y}_i^\trans (I_n + \feedt_i Q_i)G_i \err{u}_i - \err{y}_i^\trans G_i Q_i \err{y}_i - \feedt_i \err{u}_i^\trans G_i \err{u}_i \\ & = \err{y}_i^\trans \err{u}_i - \err{y}_i^\trans G_i Q_i \err{y}_i - \feedt_i \err{u}_i^\trans G_i \err{u}_i \label{eq:storage_function_quadratic_derivative_3}.
    \end{align}
    Note that $G_i Q_i$ in \eqref{eq:storage_function_quadratic_derivative_3} is positive semidefinite, which follows from simple calculations applying the Cayley-Hamilton Theorem~\cite[Th.~2.4.3.2]{horn2012matrix}. 
    Thus, from~\eqref{eq:storage_function_quadratic_derivative_3} follows the supply rate stated in the corollary. 
\end{proof}

With the last corollary, we show that the agent dynamics~\eqref{eq:agents_durchgriff} have an \gls{iseip} index greater than zero even without having a strongly convex objective function. Inspecting the \gls{oseip} and \gls{iseip} indices as a function of $\feedt_i$, we have $\ifpIdx_i(\feedt_i) = \eigmin{\feedt_i \left( I_n + \feedt_i Q_i \right)^{-1}}$ and $\ofpIdx_i(\feedt_i) = \eigmin{\left( I_n + \feedt_i Q_i \right)^{-1} Q_i }$. Due to the positive definiteness of $\left( I_n + \feedt_i Q_i \right)^{-1}$, it directly follows that $\ifpIdx_i(\feedt_i) \geq 0$ and equality holds only when $\feedt_i = 0$. Furthermore, $\ofpIdx_i(\feedt_i) \geq 0$, and $\ofpIdx_i(\feedt_i) = 0$ for all $\feedt_i$ if $Q_i$ does not have full rank, which is the case for only convex objective functions. For a strongly convex objective function, $Q_i^{-1}$ exists and we have $\ofpIdx_i(\feedt_i) = \eigmin{\left( Q_i^{-1} + \feedt_i I_n \right)^{-1}}$. For $\feedt_i \rightarrow 0$, we obtain $\ifpIdx_i(\feedt_i)\rightarrow 0$ and $\ofpIdx_i(\feedt_i)\rightarrow \eigmin{Q_i}$. For $\feedt_i \rightarrow \infty$, we obtain $\ifpIdx_i(\feedt_i)\rightarrow \frac{1}{\eigmax{Q_i}}$ and $\ofpIdx_i(\feedt_i)\rightarrow 0$ from simple limit calculations. Thus, compared to Proposition~\ref{prop:agents_oseip_iseip}, in Corollary~\ref{cor:quadratic} we do not have the upper bound $\feedt_i \leq 2$, limiting the possible \gls{oseip} and \gls{iseip} index values. Furthermore, when the objective function is only convex, the \gls{oseip} index is zero instead of negative.  

\subsection{Agent Dynamics for Constrained Optimization}

For considering constraints in the distributed optimization framework, we are inspired by the approach in~\cite{hatanaka2018passivity}. It uses a gradient descent and projected dual ascent studied in detail in~\cite{cherukuri2016asymptotic}. The method implicitly constraints the Lagrange multipliers $\lambda_{il}$ to the positive real numbers. Consider the agent system dynamics
\begin{align} \label{eq:agents_constrained}
    \Sigma_i: \left\{
    \begin{array}{cl}
    \dot{x}_i & = - \alpha_i \nabla_{x_i} L_i(x_i, \lambda_i, \mu_i) + \alpha_i u_i \\
    \dot{\lambda}_{il} & = \left\{ \begin{array}{cl}    
        0 & \text{if } \lambda_{il} = 0 \text{ and } g_{il}(x_i)<0 \\
        g_{il}(x_i) & \text{otherwise}
    \end{array} \right. \\
    \dot{\mu}_{ij} &= h_{ij}(x_i) \\
    y_i &= x_i ,
    \end{array} \right.
\end{align}
with $\alpha_i>0$. Note that system~\eqref{eq:agents_constrained} has discontinuous dynamics and the solutions are understood in the Carathéodory sense~\cite{cortes2008discontinuous}. The time derivative of a storage function $S_i(x_i,\lambda_i,\mu_i)$ may thus not exist everywhere. Therefore, we use the nonpathological derivative of~\cite[Definition~4]{bacciotti2006nonpathological}, which for a function $V(x)$ is defined as $\npd{V} = \nabla_x V^\trans \dot{x}$ in the case $V$ is continuously differentiable. Before analyzing the \gls{eip} properties, we study the properties of a steady state of~\eqref{eq:agents_constrained} in order to show that it complies with the design requirements for global optimality~\eqref{eq:condition_equilibrium_agents_constrained}. 

\begin{proposition} \label{prop:agents_constrained_steady_state}
    Consider system~\eqref{eq:agents_constrained}. Any equilibrium point $\left( \rl{x}_i, \rl{\lambda}_i, \rl{\mu}_i \right)$ fulfills the design requirements~\eqref{eq:condition_equilibrium_agents_constrained}. Further, it holds that $\lambda(t) \geq 0$ $\forall t \geq 0$ if $\lambda(0) \geq 0$.  
\end{proposition}

\begin{proof}
    Equations~\eqref{eq:condition_equilibrium_agents_constrained_1},~\eqref{eq:condition_equilibrium_agents_constrained_3} are trivially fulfilled in the equilibria of~\eqref{eq:agents_constrained}. Furthermore, the multiplier $\lambda_{il}$ is always greater or equal than zero, which can be seen by inspecting the derivative of $\lambda_{il}$ in~\eqref{eq:agents_constrained}. When the multiplier is zero, the derivative $\dot{\lambda}_{il}$ is either positive or zero, in the case $g_{il}(\rl{x}_i) < 0$. Thus, for any steady state, $\rl{\lambda}_{il} \geq 0$ must hold, which is~\eqref{eq:condition_equilibrium_agents_constrained_4}. Condition~\eqref{eq:condition_equilibrium_agents_constrained_2} also follows from the previous argumentation, since $g_{il}(x_i)>0$ implies $\lambda_{il}$ is not an equilibrium. Thus, in all equilibria, it holds either $\rl{\lambda}_{il} = 0$ and $g_{il}(\rl{x}_i) \leq 0$, or $\rl{\lambda}_{il} \geq 0$ and $g_{il}(\rl{x}_i) = 0$ per design of the algorithm~\eqref{eq:agents_constrained}, which corresponds to~\eqref{eq:condition_equilibrium_agents_constrained_5}. 
\end{proof}

After showing that all equilibria of the proposed agent dynamics for constrained optimization~\eqref{eq:agents_constrained} fulfill the global optimality design requirements~\eqref{eq:condition_equilibrium_agents_constrained}, we next investigate their \gls{eip} properties.   

\begin{proposition} \label{prop:agents_constrained}
    Consider the storage function $S_i(\err{x}_i,\err{\lambda}_i,\err{\mu}_i) = \frac{1}{2\alpha_i} \err{x}_i^\trans \err{x}_i + \frac{1}{2} \err{\lambda}_i^\trans \err{\lambda}_i + \frac{1}{2} \err{\mu}_i^\trans \err{\mu}_i$, with $\lambda_i = \col{\lambda_{il}}$ and $\mu_i = \col{\mu_{ij}}$. System~\eqref{eq:agents_constrained} is \gls{eip} if the objective function $f_i$ is convex, and \gls{oseip}($\convIdx_i$) if $f_i$ is $\convIdx_i$-strongly convex, provided that the constraints $g_{il}$ and $h_{ij}$ are convex and affine, respectively. 
\end{proposition}

\begin{proof}
    The nonpathological derivative of $S_i$ along~\eqref{eq:agents_constrained} is
    \begin{align}
        \npd{S}_i = \frac{1}{\alpha_i}\err{x}_i^\trans \dot{x}_i + \err{\lambda}_i^\trans \dot{\lambda}_i + \err{\mu}_i^\trans \dot{\mu}_i.
    \end{align}
    Inserting $\dot{x}_i$ from~\eqref{eq:agents_constrained}, we obtain
    \begin{align} \label{eq:nonpat_derivative_S_1}
        \npd{S}_i = \frac{1}{\alpha_i}\err{x}_i^\trans \big( - \alpha_i \nabla_{x_i} L_i(x_i,\lambda_i) + \alpha_i u_i \big) + \err{\lambda}_i^\trans \dot{\lambda}_i + \err{\mu}_i^\trans \dot{\mu}_i.
    \end{align}
    Adding $\nabla_{{x}_i} L_i(\rl{x}_i,\rl{\lambda}_i, \rl{\mu}_i) - \rl{u}_i$, which equals to zero, to~\eqref{eq:nonpat_derivative_S_1}, we get 
    \begin{align}
        \npd{S}_i = \err{x}_i^\trans \big( - \nabla_{x_i} \err{L}_i(x_i,\lambda_i, \mu_i) + \err{u}_i \big) + \err{\lambda}_i^\trans \dot{\lambda}_i + \err{\mu}_i^\trans \dot{\mu}_i.
    \end{align}
    Substituting the Lagrange function and expanding $\err{x}_i$ yields 
    \begin{small}
        \begin{align} \label{eq:nonpat_derivative_S}
            \npd{S}_i &= \err{x}_i^\trans \err{u}_i - \underbrace{\err{x}_i^\trans \nabla \err{f}_i(x_i,\rl{x}_i)}_{(\mathrm{A})} - \underbrace{\err{x}_i^\trans \left( \nabla g_i(x_i)^\trans \lambda_i - \nabla g_i(\rl{x}_i)^\trans \rl{\lambda}_i\right) }_{(\mathrm{B})} \nonumber \\ & 
            \quad - \underbrace{\err{x}_i^\trans \left( \nabla h_i(x_i)^\trans \mu_i - \nabla h_i(\rl{x}_i)^\trans \rl{\mu}_i \right)}_{(\mathrm{C})} + \underbrace{\err{\lambda}_i^\trans \dot{\lambda}_i}_{(\mathrm{D})} + \underbrace{\err{\mu}_i^\trans \dot{\mu}_i}_{(\mathrm{E})}. 
        \end{align}
    \end{small}
    \hspace{-0.14cm}where ${g}_i: \R^n \rightarrow \R^{\numOfIneq{i}}$ and ${h}_i: \R^n \rightarrow \R^{\numOfEq{i}}$ are the stacked functions $g_{il}$ and $h_{ij}$, i.e. ${g}_i = \col{g_{il}}$. 
    Subsequently, we show that $-\mathrm{B} + \mathrm{D }\leq 0$ and $-\,\mathrm{C} + \mathrm{E} = 0$. Hence, $\npd{S} \leq \err{x}_i^\trans \err{u}_i  - \mathrm{A}$ follows from~\eqref{eq:nonpat_derivative_S}, which is identical to~\eqref{eq:storage_function_derivative2} in Proposition~\ref{prop:agents_oseip} and thus the \gls{eip} and \gls{oseip}$(\convIdx_i)$ properties follow with the same argumentation. 
    
    We start by showing that $-\mathrm{C} + \mathrm{E} = 0$. Inserting $\dot{\mu}_i$ from~\eqref{eq:agents_constrained} into term $\mathrm{E}$ and taking into account that in any equilibrium $\rl{x}_i$, $h_i(\rl{x}_i) = 0$ holds, we have
    \begin{align}
        \err{\mu}_i^\trans h_i(x_i) &= \err{\mu}_i^\trans h_i(x_i) - \err{\mu}_i^\trans h_i(\rl{x}_i) =  \err{\mu}_i^\trans \err{h}_i(x_i,\rl{x}_i) \nonumber \\ & = \err{\mu}_i^\trans A_i \err{x}_i. \label{eq:term_E} 
    \end{align}
    Taking into account that $\nabla h_i(x_i) = A_i^\trans$, $A_i = \col{a_{ij}^\trans}$, the term $\mathrm{C}$ is identical to~\eqref{eq:term_E}, thus proving 
    \begin{align} \label{eq:terms_C_E}
        -\,\mathrm{C} + \mathrm{E} = 0.
    \end{align}
    W.r.t.\ term $\mathrm{D}$, we analyze it in both modes, i.e., when $\lambda_i = 0$ (mode 1) and $g_i(x_i) < 0$ or when $\lambda_i > 0$ (mode 2), where all operations apply component-wise. In mode 1, we have $\lambda_i = 0$ and the term $\mathrm{D}$ becomes 
    \begin{align}
        \err{\lambda}_i^\trans \dot{\lambda}_i = 0 = \lambda_i^\trans g_i(x_i).
    \end{align}
    Adding $-\rl{\lambda}_i^\trans g_i(x_i) + \rl{\lambda}_i^\trans g_i(x_i)$, which equals to zero, we obtain
    \begin{align}
        \err{\lambda}_i^\trans \dot{\lambda}_i = \err{\lambda}_i^\trans g_i(x_i) + \rl{\lambda}_i^\trans g_i(x_i) \leq \err{\lambda}_i^\trans g_i(x_i),
    \end{align}
    since $\rl{\lambda}_i \geq 0$ and $g_i(x_i) < 0$, and thus $\rl{\lambda}_i^\trans g_i(x_i) \leq 0$. In mode 2, we have 
    \begin{align}
        \err{\lambda}_i^\trans \dot{\lambda}_i = \err{\lambda}_i^\trans g_i(x_i),
    \end{align}
    and we thus conclude that 
    \begin{align} \label{eq:abschaetzung_D_1}
        \err{\lambda}_i^\trans \dot{\lambda}_i \leq \err{\lambda}_i^\trans g_i(x_i)   
    \end{align}
    always holds, regardless of the actual mode. Furthermore, adding $-\err{\lambda}_i^\trans g_i(\rl{x}_i) + \err{\lambda}_i^\trans g_i(\rl{x}_i)$ to~\eqref{eq:abschaetzung_D_1}, we have 
    \begin{align} \label{eq:abschaetzung_D_3}
        \err{\lambda}_i^\trans \dot{\lambda}_i \leq \err{\lambda}_i^\trans \err{g}_i(x_i,\rl{x}_i) + \err{\lambda}_i^\trans g_i(\rl{x}_i).
    \end{align}
    Note that $\err{\lambda}_i g_i(\rl{x}_i) = \left(\lambda_i - \rl{\lambda}_i\right) g_i(\rl{x}_i) \leq 0$, since $-\rl{\lambda}_i g_i(\rl{x}_i) = 0$ and $ \lambda_i g_i(\rl{x}_i) \leq 0$ (c.f. Proposition~\ref{prop:agents_constrained_steady_state}), and we thus obtain for~\eqref{eq:abschaetzung_D_3}
    \begin{align} \label{eq:abschaetzung_D_2}
        \err{\lambda}_i^\trans \dot{\lambda}_i \leq \left(\lambda_i - \rl{\lambda}_i \right)^\trans \err{g}_i(x_i,\rl{x}_i).
    \end{align} 
    Taking into account the fact that the functions $g_i$ are convex and using~\eqref{eq:convexity_inequalities},~\eqref{eq:abschaetzung_D_2} can be overestimated by
    \begin{align}
        \err{\lambda}_i^\trans \dot{\lambda}_i \leq \lambda_i^\trans \nabla g_i(x_i)^\trans \err{x}_i - \rl{\lambda}_i^\trans  \nabla g_i(\rl{x}_i)^\trans \err{x}_i, \nonumber
    \end{align} 
    and factorizing $\err{x}_i$ we obtain
    \begin{align}
        \err{\lambda}_i^\trans \dot{\lambda}_i \leq \err{x}_i^\trans \left( \nabla g_i(x_i) \lambda_i - \nabla g_i(\rl{x}_i) \rl{\lambda}_i \right). \nonumber
    \end{align}
    We define the right-hand side as term $\err{\mathrm{D}}$, which is equal to term $-\mathrm{B}$. Note that $\mathrm{D} \leq \err{\mathrm{D}}$ and thus
    \begin{align} \label{eq:terms_B_D}
        -\,\mathrm{B} + \mathrm{D} \leq -\,\mathrm{B} + \err{\mathrm{D}} = 0.
    \end{align}
    With~\eqref{eq:terms_C_E}~and~\eqref{eq:terms_B_D}, we have shown for~\eqref{eq:nonpat_derivative_S} that
    \begin{align} \nonumber
        \npd{S}_i \leq \err{y}_i^\trans \err{u}_i - \err{y}_i^\trans \nabla \err{f}_i(y_i,\rl{y}_i),
    \end{align}
    holds, which completes the proof. 
\end{proof}

\subsection{Controller Dynamics} 

For the controller systems, consider the integrator with feedthrough
\begin{align} \label{eq:controllers_concrete}
    \Pi_k: 
    \begin{array}{cl}
        \dot{z}_k & = \beta_k \zeta_k \\
        d_k & = z_k + \beta_k \zeta_k.
    \end{array}
\end{align}
Note that the controller dynamics~\eqref{eq:controllers_concrete} with $\beta_k>0$ are a system with integral action as in Definition~\ref{def:integral_action}, and thus $\rl{\zeta}_k = 0$, as required for an optimal steady state by Theorem~\ref{th:requirements_equilibria}~and~\ref{th:requirements_equilibria_constrained}. In addition,~\eqref{eq:controllers_concrete} is \gls{eio}. In the next proposition, the \gls{eip} properties of~\eqref{eq:controllers_concrete} are established. 

\begin{proposition} \label{prop:controllers_iseip}
    Consider the storage function $W_k(\err{z}_k) = \frac{1}{2 \beta_k} \err{z}_k^\trans \err{z}_k$. System~\eqref{eq:controllers_concrete} is \gls{iseip}($\beta_k$) for any $\beta_k> 0$. 
\end{proposition}

\begin{proof}
    Consider the time derivative of the storage function $W_k$ and insert the system dynamics of~\eqref{eq:controllers_concrete} to obtain 
    \begin{align}
        \dot{W}_k &= \frac{1}{\beta_k} \err{z}_k^\trans \dot{z}_k \\
        & = \frac{1}{\beta_k} \err{z}_k^\trans \beta_k \zeta_k. 
    \end{align}
    Next, subtracting the equilibrium $0 = \err{z}_k^\trans \rl{\zeta}_k$, we have 
    \begin{align}
        \dot{W}_k & = \err{z}_k^\trans \zeta_k - \err{z}_k^\trans \rl{\zeta}_k \nonumber \\
        & = \err{z}_k^\trans \err{\zeta}_k \nonumber \\
        & = \err{d}_k^\trans \err{\zeta}_k - \beta_k \err{\zeta}_k^\trans \err{\zeta}_k, \label{eq:storage_function_derivative_con}
    \end{align}
    where we used the output equation of the system dynamics~\eqref{eq:controllers_concrete} in the last equation. Note that~\eqref{eq:storage_function_derivative_con} verifies the \gls{iseip}($\beta_k$) property with $\beta_k > 0$.  
\end{proof}

\begin{remark} \label{rem:controllers_eip}
    If the feedthrough in the output equation of~\eqref{eq:controllers_concrete} is eliminated, \eqref{eq:controllers_concrete} is merely \gls{eip} instead of \gls{iseip}($\beta_k$), which directly follows from the proof of Proposition~\ref{prop:controllers_iseip}. 
\end{remark}

In this section, we have proposed different agent and controller dynamics and analyzed their \gls{eip} properties. The proposed agent dynamics exhibit \gls{oseip} and \gls{iseip} properties and are capable of handling private, convex constraints. Furthermore, we have shown that the design requirements of Theorems~\ref{th:requirements_equilibria}~and~\ref{th:requirements_equilibria_constrained} are fulfilled with the proposed systems, thus ensuring globally optimal equilibria. In the next sections, we use the \gls{eip} properties of the agent and controller dynamics to make convergence statements to these equilibria.


%
%


\section{Convergence Analysis with Undirected Communication Topologies} \label{sec:convergence_undirected}

In this section, we analyze the convergence of the closed-loop system in Figure~\ref{fig:burger_structure} for both unconstrained and constrained optimization with undirected communication topologies. We show that Theorem~\ref{th:requirements_convergence} can be applied with the dynamics proposed in Section~\ref{sec:agent_controller_dynamics}. Furthermore, we demonstrate that less restrictive convergence requirements can be obtained when the admissible agent and controller dynamics are narrowed to the ones proposed in Section~\ref{sec:agent_controller_dynamics}.

%
%

\subsection{Unconstrained Optimization}


In Section~\ref{sec:agent_controller_dynamics}, we have shown that the agent dynamics~\eqref{eq:agents_concrete}, \eqref{eq:agents_durchgriff}, and~\eqref{eq:agents_constrained} are \gls{oseip} with an excess of passivity if the objective function of the agent is strongly convex (see Propositions~\ref{prop:agents_oseip},~\ref{prop:agents_oseip_iseip} and~\ref{prop:agents_constrained}).\footnote{Note that the assumption of strongly convex functions is widely adopted in the literature (see, e.g.,~\cite{li2020input,li2020distributed,kia2015distributed}).} In this case, with any of the controller systems proposed in Section~\ref{sec:agent_controller_dynamics}, convergence to the global optimizer directly follows from Theorem~\ref{th:requirements_convergence}. In some applications, however, the agents may only have convex objective functions, thus having passive but no \gls{oseip} dynamics according to Propositions~\ref{prop:agents_oseip},~\ref{prop:agents_oseip_iseip} and~\ref{prop:agents_constrained}. 
In the following theorem, we show that for the unconstrained problem and an undirected communication topology, the proposed agent and controller dynamics from Section~\ref{sec:agent_controller_dynamics} achieve convergence, if only the sum of all objective functions is strictly convex. 
	

\begin{theorem} \label{th:stability}
    Let the agents $i \in \setOfAgents$ and controller systems $k \in \setOfControllers$ obey the dynamics~\eqref{eq:agents_concrete} and~\eqref{eq:controllers_concrete}, respectively, and let Propositions~\ref{prop:agents_oseip} and~\ref{prop:controllers_iseip} hold. Further, let the agents be interconnected by a symmetric communication topology~\eqref{eq:interconnection} fulfilling property~\eqref{eq:nullspace_property}. Then, all trajectories converge to an equilibrium point in the invariant manifold $\eqM$ as specified in Theorem~\ref{th:requirements_equilibria}, if the objective functions of the agents $f_i:\R^n \rightarrow \R$ are convex and the sum of all objective functions is strictly convex. 
\end{theorem}

\begin{proof}
    Consider the network-wide Lyapunov function~\eqref{eq:lyapunov} and the time derivative
    \begin{align}
        \dot{V} & \leq \sum_{i \in \setOfAgents} \left( \err{u}_i^\trans \err{y}_i - \Psi_i(\err{y}_i) \right) + \sum_{k \in \setOfControllers} \left( \err{\zeta}_k^\trans \err{d}_k - \ifpIdx_k \zeta_k^\trans \zeta_k \right). \nonumber
    \end{align}
    Written in network variables, we have
    \begin{align}
        \dot{V} &\leq \err{u}^\trans \err{y} - \sum_{i \in \setOfAgents} \Psi_i(\err{y}_i) + \err{\zeta}^\trans \err{d}  - \err{\zeta}^\trans \left(\diag{\ifpIdx} \otimes I_n\right) \err{\zeta} \nonumber \\
        & = - \sum_{i \in \setOfAgents}\Psi_i(\err{y}_i)  - \err{y}^\trans K \err{y} , \label{eq:lyapunov_derivative4}
    \end{align}
    where the last equality follows by inserting the interconnection topology~\eqref{eq:interconnection} into the first, third and fourth term, and introducing $K = \left( \inci \otimes I_n \right) \left(\diag{\ifpIdx}\otimes I_n\right) \left( \inci \otimes I_n \right)^\trans$. The first term in~\eqref{eq:lyapunov_derivative4} is positive semidefinite, since every $\Psi_i(\err{y}_i)$ is nonnegative when $f_i:\R^n \rightarrow \R$ is convex, see Corollary~\ref{cor:agents_eid}. Furthermore, the first term in~\eqref{eq:lyapunov_derivative4} is positive definite when all $\err{y}_i$, $i \in \setOfAgents$ are in consensus. To show that, assume that the agents are in consensus, i.e., $\err{y}_i = \err{y}_{\text{c}}$, and since $\opt{\yrm}$ is unique due to strict convexity, it holds that
    \begin{align}
        \sum_{i\in\setOfAgents} \Psi_i(\err{y}_i) &=  \err{y}_{\text{c}}^\trans (\sum_{i\in\setOfAgents}\nabla f_i(y_{\text{c}}) - \sum_{i\in\setOfAgents}\nabla f_i(\rl{y}_{\text{c}})), \nonumber 
    \end{align}
    which is positive for all $y_{\text{c}}$ other than $\opt{\yrm}$, if the sum of all objective functions is strictly convex (see condition~\eqref{eq:convexity_differentiable}), as required in the theorem.  The second term $\err{y}^\trans K \err{y}$ in~\eqref{eq:lyapunov_derivative4} is nonnegative, and positive whenever $\err{y}_i$ are not in consensus. This is because the matrix $K$ is positive semidefinite, and $\err{y}^\trans K \err{y} = 0$ is equivalent to $\left( \inci \otimes I_n \right)^\trans \err{y} = 0$ \cite[Observation~7.1.6]{horn2012matrix}. With Proposition~\ref{prop:kernel_kronecker}, it holds that $\err{y}^\trans K \err{y} = 0$ if and only if $\err{y}_i$, $i \in \setOfAgents$ are in consensus. Thus,~\eqref{eq:lyapunov_derivative4} is negative semidefinite and zero only if the agents are at the optimizer and at consensus, i.e., $y_i = \opt{\yrm}$, $i \in \setOfAgents$, which also is the definition of the manifold $\eqM$ in Theorem~\ref{th:requirements_equilibria}. Thus, we conclude that the output $y_i$ of all agents converge to an equilibrium in $\eqM$. Furthermore, with \gls{eio} of the $\Sigma_i$ systems, the states $x_i$ are also constant. Since in the equilibrium manifold $\eqM$ the outputs of the agents $y_i$ are in consensus, it holds with~\eqref{eq:interconnection1} and Proposition~\ref{prop:kernel_kronecker} in the Appendix~\ref{sec:appendix_proofs} that $\zeta_k = 0$, and we have $\dot{z}_k = 0$, which means that all $\rl{z}_k$, and thus $\rl{d}_k$, are constant. 
\end{proof}



\subsection{Constrained Optimization} \label{sec:constrained_optimization}



In this subsection, we investigate the convergence behavior for constrained optimization using the discontinuous agent dynamics~\eqref{eq:agents_constrained} with undirected communication structures. In this case, Theorem~\ref{th:requirements_convergence} is only applicable under Assumption~\ref{assum:existence_solution}. Thus, we first give the closed-loop networked system (see Figure~\ref{fig:burger_structure}) and ensure that appropriate solutions exists, i.e., that Assumption~\ref{assum:existence_solution} holds. Afterwards, we follow similar steps as in Theorem~\ref{th:stability} for the unconstrained case and show via Lyapunov theory that the convergence of all solutions to the global optimizer $\opt{\yrm}$ requires only the sum of all objective functions to be strictly convex.


The closed-loop system is composed of agent systems~\eqref{eq:agents_constrained} and controller systems~\eqref{eq:controllers_concrete} interconnected by a generalized communication structure~\eqref{eq:communication_bipartite_incidence}, i.e.,
\begin{subequations}\label{eq:closed_loop_discontinuous}
    \begin{align}
        \dot{x} & = - \alpha \nabla_{x} L(x,\lambda,\mu) - \alpha (R\otimes I_n) d \\
        \dot{\lambda} & = \left\{ \begin{array}{cl}    
            0 & \text{if } \lambda = 0 \text{ and } g(x)<0 \label{eq:multiplier_discontnuity}\\
            g(x) & \text{otherwise}
        \end{array} \right. \\
        \dot{\mu} &= h(x) \\
        \dot{z} &= \beta (R\otimes I_n)^\trans y, 
    \end{align} 
\end{subequations}
with the outputs $y = x$ and $d = z + \beta \left( R \otimes I_n \right)^\trans y$. The variables $x$, $\lambda$, $\mu$, $z$, and functions $L:\R^{\numOfAgents n} \times \R^{ \numOfEq{}} \times \R^{ \numOfIneq{}} \rightarrow \R^{\numOfAgents n}$, $g:\R^{\numOfAgents n} \rightarrow \R^{ \numOfIneq{}}$, $h:\R^{\numOfAgents n} \rightarrow \R^{\numOfEq{}}$ are network variables and functions, respectively, i.e., $x = \col{x_i} \in \R^{\numOfAgents n}$ and $L = \col{L_i}$, $i\in\setOfAgents$. The variables $\alpha$ and $\beta$ are diagonal matrices, i.e., $\alpha = \diag{\alpha_i} \otimes I_n$. In~\eqref{eq:multiplier_discontnuity}, the operations apply component-wise. 
Note that the closed-loop system is a dynamical system with discontinuous right-hand side due to~\eqref{eq:multiplier_discontnuity}. For such systems, the existence of a solution is not ensured by classical theorems, since the vector field is not locally Lipschitz. 
Instead, the existence of solutions has to be studied for each specific system before the properties of their solutions are analyzed~\cite{cortes2008discontinuous}.
In the following, we proof the existence, uniqueness and continuity w.r.t.\ initial states of the solution of \eqref{eq:closed_loop_discontinuous}. For that, we follow similar arguments as in~\cite{cherukuri2016asymptotic} and make use of the theory of projected dynamical systems, which are a special type of dynamical systems with discontinuous right-hand side. For projected dynamical systems, there exist useful results about the existence of solutions, see~Appendix~\ref{sec:appendix_projected}. The following proposition shows that the closed-loop system~\eqref{eq:closed_loop_discontinuous} can be cast as such a projected dynamical system. 
\begin{proposition} \label{prop:closed_loop_as_projected_sys}
    The discontinuous system~\eqref{eq:closed_loop_discontinuous} can be written as a projected dynamical system. 
\end{proposition}
\begin{proof}
    The proof consists of finding a continuous vector field and a closed convex set $\mathcal{K}$ such that the projected dynamical system defined thereby is equal to system~\eqref{eq:closed_loop_discontinuous}. Consider the vector field 
    \begin{align} \label{eq:continuous_vector_field}
        \phi(x,\lambda,\mu,z) = \mat{- \alpha \nabla_{x} L - \alpha (R\otimes I_n) z \\ g(x) \\ h(x) \\ (R\otimes I_n)^\trans x},
    \end{align}
    which is a version of~\eqref{eq:closed_loop_discontinuous} without the discontinuity at $\lambda = 0$, and the closed and convex set $\mathcal{K} = \R^{\numOfAgents n} \times \R^{\numOfIneq{}}_{\geq 0} \times \R^{\numOfEq{}} \times \R^{\numOfControllers n}$. Observe that the set $\mathcal{K}$ excludes negative Lagrange multipliers $\lambda$, understood component-wise. Next, consider the projected dynamical system 
    \begin{align}\label{eq:projected_system_closed_loop}
        \mat{\dot{x}^{\diamond} \\ \dot{\lambda}^\diamond \\ \dot{\mu}^\diamond \\ \dot{z}^\diamond} = \vproj{\phi((x^\diamond,\lambda^\diamond,\mu^\diamond,z^\diamond))}{(x^\diamond,\lambda^\diamond,\mu^\diamond,z^\diamond)}{\mathcal{K}}. 
    \end{align}  
    Note that as per Proposition~\ref{prop:geometrical_interpretation_vector_projection}, Case (i), in the interior of $\mathcal{K}$, i.e., when $\lambda^\diamond > 0$, the projection is not active and~\eqref{eq:projected_system_closed_loop} is identical to~\eqref{eq:closed_loop_discontinuous}, i.e., $\dot{\lambda}^\diamond = g(x^\diamond)$. Thus, we solely need to investigate~\eqref{eq:projected_system_closed_loop} at the boundary of $\mathcal{K}$ and compare it with~\eqref{eq:closed_loop_discontinuous}. Note that the boundary of $\mathcal{K}$ is also where the discontinuity in~\eqref{eq:closed_loop_discontinuous} is.
    On the boundary of $\mathcal{K}$, some Lagrange multipliers are zero, i.e., $\lambda_q^\diamond = 0$, with $q\in \mathcal{I}$, where $\mathcal{I} \subseteq \setOfIneq{}$ is the set of multipliers that are zero. For these multipliers on the boundary of $\mathcal{K}$, we have either $g_q(x^\diamond) \geq 0$ or $g_q(x^\diamond)<0$. In case of $g_q(x^\diamond) \geq 0$, note that the vector field points inwards into $\mathcal{K}$ and thus the projection is not active. In~\eqref{eq:closed_loop_discontinuous}, the switch is not active, and we thus have $\dot{\lambda}_q^\diamond = g_q(x^\diamond)$ in~\eqref{eq:projected_system_closed_loop} and $\dot{\lambda}_q = g_q(x)$ in~\eqref{eq:closed_loop_discontinuous}, which is equal in both systems. When $g_q(x^\diamond)<0$, note that the vector field points outwards perpendicular to the boundary of $\mathcal{K}$, and due to Remark~\ref{rem:geometrical_interpretation}, in such a case the vector projection onto $\mathcal{K}$ is zero, i.e., $\dot{\lambda}^\diamond_q = 0$. For~\eqref{eq:closed_loop_discontinuous}, in case that $g_q(x)<0$, we have $\dot{\lambda}_q = 0$, which is again identical to the projected system. This completes the proof. 
\end{proof}

With Proposition~\ref{prop:closed_loop_as_projected_sys}, Proposition~\ref{prop:properties_projected_sys} in Appendix~\ref{sec:appendix_projected}, and by assuming that all $f$, $g$ and $h$ are Lipschitz, we conclude that the closed-loop system~\eqref{eq:closed_loop_discontinuous} has a unique continuous solution which is continuous w.r.t.\ the initial state. Consequently, we can now study the stability properties of the system trajectories. For that, we use the Lyapunov theory for systems with discontinuous right-hand side, in particular the invariance principle for Carathéodory systems proposed in~\cite{bacciotti2006nonpathological}. 

\begin{theorem} \label{th:stability_constrained}
    Let the agents and controller systems obey the dynamics~\eqref{eq:agents_constrained} and~\eqref{eq:controllers_concrete}, respectively, let Propositions~\ref{prop:agents_constrained} and~\ref{prop:controllers_iseip} hold, and the systems be interconnected by~\eqref{eq:interconnection}. Then, the equilibrium $\rl{y}_i = \opt{\yrm}$, $i\in\setOfAgents$ is globally asymptotically stable, if the sum of all objective functions $f_i:\R^n \rightarrow \R$, $i\in \setOfAgents$ is strictly convex and $\lambda_{il}(0)\geq 0$. 
\end{theorem}

\begin{proof}
    Consider the Lyapunov function~\eqref{eq:lyapunov} and its time derivative
    \begin{align} \label{eq:lyapunov_npd}
        \npd{V} = \sum_{i \in \setOfAgents} \npd{S}_i + \sum_{k \in \setOfControllers} \dot{W}_k 
    \end{align}
    which holds for almost every $t \in [0,\infty)$, \cite{bacciotti2006nonpathological}. Note that the agent dynamics~\eqref{eq:agents_constrained} are dissipative w.r.t.\ the same supply rate as the agent dynamics~\eqref{eq:agents_concrete} as per Proposition~\ref{prop:agents_constrained}. Thus, with identical arguments as in the derivation of the inequality~\eqref{eq:lyapunov_derivative4}, we have for~\eqref{eq:lyapunov_npd}
    \begin{align} \label{eq:lyapunov_npd_2}
        \npd{V} = \sum_{i \in \setOfAgents} \npd{S}_i + \sum_{k \in \setOfControllers} \dot{W}_k \leq - \sum_{i \in \setOfAgents}\Psi_i(\err{y}_i)  - \err{y}^\trans K \err{y}, 
    \end{align}
    where $K = \left( \inci \otimes I_n \right) \left(\diag{\ifpIdx}\otimes I_n\right) \left( \inci \otimes I_n \right)^\trans$.
    With the invariance principle in~\cite[Proposition~3]{bacciotti2006nonpathological}, all trajectories converge to the largest invariant set in the set $\{ (x,\lambda,\mu,z) \in \R^{\numOfAgents n \times \numOfControllers n \times \numOfEq{} \times \numOfIneq{}} \mid \npd{V}=0 \}$. With the same reasoning as in Theorem~\ref{th:stability}, we conclude that all trajectories are bounded and converge to $\eqM$. Since the sum of all objective functions is strictly convex, the minimizer is unique \cite{boyd2004convex} and thus $\lambda = \lambda^*$ and $\mu = \mu^*$ in steady state. Since all $\err{y}_i$ are in consensus, the states $z$ and outputs $d$ are constant. 
\end{proof}

Note that inequality~\eqref{eq:lyapunov_npd_2} is fulfilled when using any agent dynamics~\eqref{eq:agents_concrete}, \eqref{eq:agents_durchgriff} and~\eqref{eq:agents_constrained}, since they have structurally identical \gls{eip} properties. Thus, convergence is ensured even when agents use different system dynamics~\eqref{eq:agents_concrete},~\eqref{eq:agents_durchgriff}~or~\eqref{eq:agents_constrained} with convex functions, or other dynamics as long as these dynamics are \gls{oseip}($\ofpIdx_i$) and the design requirements from Theorem~\ref{th:requirements_equilibria}~or~\ref{th:requirements_equilibria_constrained} hold.

\section{Convergence Analysis with Directed Communication Topologies} \label{sec:convergence_directed}

In this section, we consider the unconstrained problem in which agents and controllers are connected via a directed communication structure. First, we show with \autoref{} that such a directed communication structure hampers obtaining convergence results as in the previous sections. Afterwards, we obtain convergence results that require, in contrast to the case of undirected communication, verifying a global linear matrix inequality. 

\begin{proposition}
    Consider a set $\setOfAgents$ of agents $\Sigma_i$, $i\in\setOfAgents$ and a set $\setOfControllers$ of controllers $\Pi_k$, $k\in\setOfControllers$. Assume the agent dynamics are \gls{oseip}($\ofpIdx_i$) and the controller dynamics \gls{iseip}($\ifpIdx_k$). Consider a Lyapunov function~\eqref{eq:lyapunov} as in Theorem~\ref{th:stability}. Then, this Lyapunov fails to prove asymptotic stability. 
\end{proposition}
\begin{proof}
    Using the \gls{oseip}($\ofpIdx_i$) and \gls{iseip}($\ifpIdx_k$) properties, the time derivative of the Lyapunov function in network variables is
    \begin{align}
        \dot{V} &\leq \err{u}^\trans \err{y} - \err{y}^\trans (\ofpIdx_\setOfAgents \otimes I_n) \err{y} + \err{\zeta}^\trans \err{d} - \err{\zeta}^\trans (\ifpIdx_\setOfControllers \otimes I_n) \err{\zeta}, \label{eq:lyapunov_derivative_directed}
    \end{align}
    with $\ofpIdx_\setOfAgents = \diag{\ofpIdx_i}$ and $\ifpIdx_\setOfControllers = \diag{\ifpIdx_k}$.
    By inserting the directed interconnection structure~\eqref{eq:interconnection_nonsymmetric} in \eqref{eq:lyapunov_derivative_directed}, we obtain
    \begin{align}\label{eq:lyapunov_derivative_directed2}
        \dot{V} \leq -\frac{1}{2}\mat{\err{y} \\ \err{d}}^\trans \mat{2 K &\left(R_\setOfControllers - R_\setOfAgents\right) \\ \left(\left(R_\setOfControllers - R_\setOfAgents\right) \right)^\trans & 0} \otimes I_n \mat{\err{y} \\ \err{d}}, 
    \end{align}
    with $K = \ofpIdx_\setOfAgents + R_\setOfAgents \ifpIdx_\setOfControllers R_\setOfAgents^\trans$. Note that the matrix in~\eqref{eq:lyapunov_derivative_directed2} is indefinite~\cite{benzi2005numerical}, and thus fails to prove asymptotic stability.\footnote{Note that the Kronecker product with $I_n$ does not alter the definiteness of the matrix.}
\end{proof}

To provide remedy, the zero block in the bottom right of the matrix in~\eqref{eq:lyapunov_derivative_directed2} needs to be replaced by at least a positive semidefinite matrix in order to achieve positive semidefiniteness. This can be achieved by either requiring the agent dynamics to exhibit, in addition to the \gls{oseip}($\ofpIdx_i$)  property, an \gls{iseip}($\ifpIdx_i$) property, or by requiring the controller dynamics to be \gls{oseip}($\ofpIdx_k$) in addition to \gls{iseip}$(\ifpIdx_k)$. 
Then, the matrix becomes
\begin{align}\label{eq:matrix_definite}
    \mat{2 K & \left(R_\setOfControllers - R_\setOfAgents\right) \\ \left(R_\setOfControllers - R_\setOfAgents\right)^\trans & 2 L},
\end{align}
with $L = \ofpIdx_\setOfControllers + R_\setOfControllers^\trans \ifpIdx_\setOfAgents R_\setOfControllers$. 
Since an \gls{oseip}($\ofpIdx_k$) property for the controller systems would prevent the systems of having integral action as in Definition~\ref{def:integral_action}, which is instrumental for achieving global optimality, the only option is to require the agent dynamics to be also \gls{iseip}($\ifpIdx_i$). In~\eqref{eq:agents_durchgriff}, a system which is \isoseip$(\ifpIdx_i, \ofpIdx_i)$ is proposed. This system is used in the next theorem to establish convergence to the global optimizer. Consider a group of agents~\eqref{eq:agents_durchgriff} and controllers~\eqref{eq:controllers_concrete} without feedthrough as in Remark~\ref{rem:controllers_eip} interconnected by a directed, non-symmetric communication structure~\eqref{eq:interconnection_nonsymmetric}
\begin{subequations}\label{eq:closed_loop_directed}
    \begin{align} 
        \dot{x} & = -\alpha \nabla f(x + \feedt u) - \alpha \left( R_\setOfControllers \otimes I_n \right)d \\
        \dot{z} & = \beta \left(R_\setOfAgents \otimes I_n\right)^\trans y, \label{eq:closed_loop_directed_2}
    \end{align}
\end{subequations}
with $y = x + \feedt u$, $\gamma = \diag{\gamma_i} \otimes I_n$, $d = z$ and $u = -\left( R_\setOfControllers \otimes I_n \right) d$, and $\alpha, \beta$ as in~\eqref{eq:closed_loop_discontinuous}.

\begin{theorem} \label{th:stability_directed}
    Consider the closed-loop system~\eqref{eq:closed_loop_directed}. Let the dynamics of the agents $i \in \setOfAgents$ and controllers $k\in\setOfControllers$ fulfill Propositions~\ref{prop:agents_oseip_iseip} and~\ref{prop:controllers_iseip}, respectively. Assume that the matrices $R_\setOfAgents$ and $R_\setOfControllers$ of the non-symmetric communication structure~\eqref{eq:interconnection_nonsymmetric} fulfill property~\eqref{eq:nullspace_property}. Then, the equilibrium $\rl{y}_i = \opt{\yrm}$, $i\in\setOfAgents$ is a globally asymptotically stable equilibrium point of the closed-loop system~\eqref{eq:closed_loop_directed}, if the singular matrix 
    \begin{equation} \label{eq:condition1_theorem_directed}
        \mat{2 K & \left(R_\setOfControllers - R_\setOfAgents\right) \\ \left(R_\setOfControllers - R_\setOfAgents\right)^\trans & 2 L},
    \end{equation} 
    with $K = \ofpIdx_\setOfAgents$ and $L = R_\setOfControllers^\trans {\ifpIdx}_\setOfAgents R_\setOfControllers$, is positive semidefinite and the matrix $\beta_\setOfControllers R_\setOfAgents^\trans K^{-1} \tilde{R}$, with $\tilde{R} = R_\setOfControllers - R_\setOfAgents$ and $\beta_\setOfControllers = \diag{\beta_k}$, has eigenvalues which are either zero or have a positive real part. 
\end{theorem} 

\begin{proof}
    Since the agent dynamics~\eqref{eq:agents_durchgriff} are \isoseip($\ifpIdx_i,\ofpIdx_i$) as in Proposition~\ref{prop:agents_oseip_iseip} and the controller dynamics~\eqref{eq:controllers_concrete} are \gls{eip}, the derivative of the Lyapunov function~\eqref{eq:lyapunov} is
    \begin{align} \label{eq:lyapunov_derivative_directed3}
            \dot{V} \leq -\frac{1}{2}\mat{\err{y} \\ \err{d}}^\trans \mat{2 K & R_\setOfControllers - R_\setOfAgents \\ \left(R_\setOfControllers - R_\setOfAgents\right)^\trans & 2 L} \otimes I_n \mat{\err{y} \\ \err{d}},
    \end{align}
    with $K$ and $L$ as stated in the theorem. Since it is required by the theorem that the matrix in~\eqref{eq:lyapunov_derivative_directed3} is positive semidefinite, $\dot{V}\leq 0$ holds and boundedness of the trajectories is guaranteed. For proving that, all trajectories converge to the set~$\eqM$, we invoke Lasalle's Invariance Principle. The proof consists of showing that the trajectory converges to a point, i.e., there are no periodic oscillations. Since all equilibria are in the set $\eqM$ as per Theorem~\eqref{th:requirements_equilibria}, it suffices to show that convergence is to a point. 
    
    Denote $v^\trans = (\err{y}^\trans, \err{d}^\trans)$ and $U$ the matrix in~\eqref{eq:lyapunov_derivative_directed3}. Note that a vector $v \in \R^{(\numOfAgents+\numOfControllers)n}$ corresponds to the set $\mathcal{L} = \{(x,z) \in \R^{(\numOfAgents+\numOfControllers)n} \mid \dot{V} = 0\}$ if and only if $v^\trans U v = 0$ holds. To inspect which vectors $v$ belong to the set $\mathcal{L}$, we compute 
    \begin{align}\label{eq:set_invariance}
        \frac{1}{2}v^\trans U v = \err{y}^\trans \mathsf{K} \err{y} + \err{y}^\trans \left(\left(R_\setOfControllers - R_\setOfAgents\right) \otimes I_n \right) \err{d} + \err{d}^\trans \mathsf{L} \err{d},
    \end{align}
    where $\mathsf{K} = K \otimes I_n$ and $\mathsf{L} = L \otimes I_n$ for convenience. With \cite[Observation~7.1.6]{horn2012matrix}, it holds that $v^\trans U v = 0$ if and only if $Uv = 0$. The first equation in $Uv = 0$, i.e,
    \begin{align} \label{eq:lyapunov_derivative_directed_zero1}
        \mathsf{K} \err{y} + \left( \left(R_\setOfControllers - R_\setOfAgents\right) \otimes I_n \right) \err{d} = 0,
    \end{align}
    relates the output of agents and controllers in the set $\mathcal{L}$. Solving~\eqref{eq:lyapunov_derivative_directed_zero1} w.r.t.\ $\err{y}$ and inserting it in the system dynamics~\eqref{eq:closed_loop_directed_2} of the controllers, we have
    \begin{align} \label{eq:dynamics_in_lasalle}
        \dot{\err{d}} = \dot{\err{z}} = -\beta \left(R_\setOfAgents^\trans \otimes I_n \right) \mathsf{K}^{-1} \left( \left(R_\setOfControllers - R_\setOfAgents\right) \otimes I_n \right) \err{d},
    \end{align}
    which describes the dynamics in $\mathcal{L}$. Thus, it suffices to show that the output $d$ and state $z$ in~\eqref{eq:dynamics_in_lasalle} converge to a steady state. Note that the system matrix in~\eqref{eq:dynamics_in_lasalle} has not full rank and thus cannot be Hurwitz. However, if the matrix in~\eqref{eq:dynamics_in_lasalle} has only eigenvalues at the origin or with a positive real part, all trajectories $d(t)$ converge to a steady state. Thus, the input of the agent systems~\eqref{eq:agents_durchgriff} is constant, and the system reaches a steady state. The condition in the theorem that $\beta_\setOfControllers R_\setOfAgents^\trans K^{-1} \tilde{R}$ has only eigenvalues at the origin or with a positive real part implies that the matrix in~\eqref{eq:dynamics_in_lasalle} has also only eigenvalues at the origin or with a positive real part, since they are related by the Kronecker product with the unity matrix. Since the conditions of Theorem~\ref{th:requirements_equilibria} are fulfilled, all equilibria correspond to the global minimizer $\rl{y}_i = \opt{\yrm}$, which concludes the proof.
\end{proof}

\begin{remark}
    For undirected communication we have $R_\setOfAgents = R_\setOfControllers$ and thus $\err{R} = 0$, yielding $\dot{\err{d}} = 0$. Thus, the conditions in Theorem~\ref{th:stability_directed} of zero eigenvalues or with positive real part are automatically fulfilled, and we recover the results of Theorem~\ref{th:stability}. 
\end{remark}

With the following corollary, we can drop the second condition in Theorem~\ref{th:stability_directed} when a certain number of controllers are used. 

\begin{corollary}
    Consider the setup of Theorem~\ref{th:stability_directed}. Then, the equilibrium $\rl{y}_i = \opt{\yrm}$, $i\in\setOfAgents$ is a globally exponentially stable equilibrium point of the closed-loop system if exactly $\numOfControllers = \numOfAgents-1$ controllers are used and the matrix 
    \begin{align} \label{eq:condition_corollary_directed}
        \mat{2 K & R_\setOfControllers - R_\setOfAgents \\ \left(R_\setOfControllers - R_\setOfAgents\right)^\trans & 2 L},
    \end{align} 
    with $K$ and $L$ as in~\eqref{eq:condition1_theorem_directed} is positive definite. 
\end{corollary}

\begin{proof}
    Consider the same Lyapunov function as in Theorem~\ref{th:stability_directed}. Note that the matrices $K$ and $L$ in~\eqref{eq:lyapunov_derivative_directed} are, in general, positive definite and positive semidefinite, respectively. If $\numOfControllers = \numOfAgents-1$, matrix $L = R_\setOfControllers^\trans \ifpIdx_\setOfControllers  R_\setOfControllers $ is positive definite since $R_\setOfControllers$ has full column rank. Thus, depending on the communication topology, the matrix~\eqref{eq:condition_corollary_directed} may be positive definite, which is set as a condition in the corollary.
\end{proof}

Note that both results for directed communication structures require, in contrast to the case with undirected topologies, to check global LMIs of the size $\numOfAgents+\numOfControllers$, i.e., the added number of agents and controllers. While such a check is numerically cheap and feasible for small networks, it becomes costly for larger networks with many participants. Thus, undirected communication topologies seem preferable for large-scale networks.

\section{Numerical Results} \label{sec:results}

In this section, we give a numerical example to illustrate some of the theoretical results obtained throughout the paper. In particular, we put an emphasis on showing that the agents may use different optimization algorithms and may leave the optimization without compromising global optimality and convergence nor requiring a global initialization. 

We consider 100 agents with scalar objective functions $f_i:\R \rightarrow \R$ and constraints randomly generated out of the three possible objective function
\begin{align}
    \text{Model 1:} \quad & \begin{array}{l} f_i(\yrm) = a_i \yrm^2 + b_i \yrm\end{array}  \\
    \text{Model 2:} \quad &\begin{array}{l} f_i(\yrm) = a_i \yrm^2 + b_i \yrm \\ \text{s.t. } \yrm \leq 0.5   \end{array} \\
    \text{Model 3:} \quad &\begin{array}{l} f_i(\yrm) =  e^{\yrm + b_i} + e^{-(\yrm + b_i)}  \end{array} .
\end{align}
The parameters $a_i \in [0, 2]$ and $b_i \in [-2, 2]$ are randomly generated with uniform distributions. Note that the objective function models 1 and 2 have an optimizer at $\opt{\yrm} = -\frac{b_i}{2a_i}$, and model 3 at $\opt{\yrm} = -b_i$. Since all $b_i$ are uniformly distributed with a mean at zero, it is to be expected that the $\numOfAgents$ agents have an optimizer near to zero.

Further, we assume each agent has a controller with $\beta_k=35$. The agents use heterogeneous dynamics, i.e.,~\eqref{eq:agents_concrete} or~\eqref{eq:agents_durchgriff}, randomly, for the unconstrained model 1 and 3. Dynamics~\eqref{eq:agents_constrained} are used for all agents having model 2. The parameters of the agent dynamics are set to $\alpha_i = 1$ and $\feedt_i = 1$. For the generalized communication structure~\eqref{eq:communication_bipartite_incidence}, a random matrix fulfilling property~\eqref{eq:nullspace_property} is generated. The probability of an agent $i \in \setOfAgents$ to exchange information with agent $j \in \setOfAgents$ is 10\%, such that every agent is on average connected to 10 other agents. At time $t = \SI{100}{\second}$, the group of 100 agents is separated into two groups of 50 agents, which form two independent groups. The groups are chosen such that the 50 agents with the greatest $b_i$, roughly speaking with $b_i \in [0, 2]$, are in a group. Since models 1 and 2 have on average in this group an optimizer at $\opt{\yrm}=-0.5$ and model 3 at $\opt{\yrm}=-1$, we expect an optimizer at around $-0.75$ for this group. Conversely, the 50 agents with the smallest $b_i$, roughly speaking with $b_i \in [-2, 0]$, are in the second group. Analogously, for this group, an optimizer at $0.75$ is expected, which is restricted to $0.5$ due to the constraint of model 2. Then, two connected groups of agents remain in individual distributed optimizations. The probability to be connected with another agent remains 10\%, such that every agent is on average connected to 5 agents of its group. 

In Figure~\ref{fig:estimations}, the local estimations of the agents are shown. Until time $t = \SI{100}{\second}$, all agents are connected and the global minimizer $\opt{\yrm} = 0.0258$ is found cooperatively. At time $t = \SI{100}{\second}$, the group splits in two, e.g., due to network failure of critical links. However, the two groups find cooperatively their respective global minima of $\opt{\yrm}_1 = 0.5$ and $\opt{\yrm}_2 = -0.7789$. In Figure~\ref{fig:error}, the exponential convergence (linear in a logarithmic scale) is shown. The exponential convergence for similar algorithms was already established in~\cite{kia2015distributed,shi2015extra}. These methods, however, require a global initialization every time an agent leaves the network. Our method does not require any global initialization at any time, provided that the agent dynamics fulfill the design requirements for global optimality of Theorem~\ref{th:requirements_equilibria}~and~\ref{th:requirements_equilibria_constrained}, which is one of the main contributions of the work at hand. Furthermore, we show in the example that ensuring global optimality and convergence in distributed optimization is possible even when agents use heterogeneous optimization dynamics.   

\begin{figure}[t]
    \centering
    \includegraphics{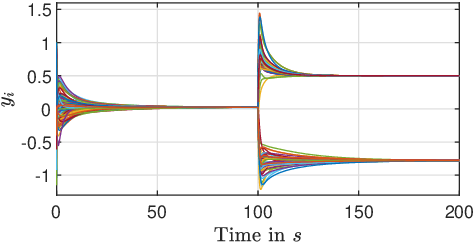}
    \caption{Estimations of the global optimizer of all agents}
    \label{fig:estimations}
\end{figure}

\begin{figure}[t]
    \centering
    \includegraphics{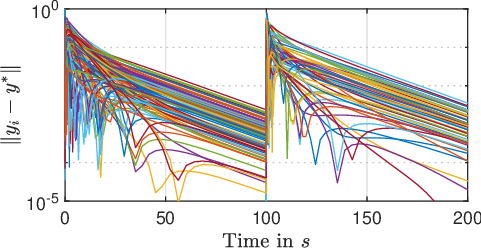}
    \caption{Error of the estimations of all agents w.r.t.\ to the global optimizer in logarithmic scale}
    \label{fig:error}
\end{figure}

\section{Conclusion} \label{sec:conclusion}

In this paper, we have presented a novel passivity-based perspective for distributed optimization algorithms, composed of agent and controller systems. We derive local design requirements for the agent and controller systems ensuring that the distributed optimization always converges to the global optimizer. Therefore, in contrast to the literature, we do not focus on any specific algorithm which all agents have to follow, but derive conditions on the agent dynamics to achieve global optimality and convergence. In particular, the approach works without any global initialization and by exchanging only a single variable. Thus, the agents may leave and rejoin the distributed optimization without compromising global optimality and convergence. 
%
%
Future work will study discrete-time communication allowing a robust practical implementation and the EIP analysis of different agent and controller dynamics to allow more heterogeneous dynamics.

\printbibliography

\appendix

\subsection{Properties of the Kronecker product} \label{sec:appendix_kronecker}

The following Kronecker product properties~\cite[p.~107]{eves1980elementary} are used in some proofs throughout the paper. Consider matrices $A, B, C$ and $D$ of appropriate dimension. 

\begin{description}
    \item[(P1)] The mixed-product property:  
    \begin{equation}
        (A \otimes B)(C \otimes D) = AC \otimes BD.
    \end{equation}
    \item[(P2)] Associativity:  
    \begin{equation}
        A \otimes (B + C) = A \otimes B + A \otimes C.
    \end{equation} 
\end{description}

\subsection{Properties of projected dynamical systems} \label{sec:appendix_projected}

Projected dynamical system theory use the concept of projecting a vector onto a closed convex set~\cite{nagurney2012projected}. They make use of the common (point) projection of point $z \in \R^n$ onto the closed convex set $\mathcal{K}$ defined as $\proj{z}{\mathcal{K}} = \arg \min_{x \in \mathcal{K}} \Vert x-z \Vert$. The vector projection is defined as~\cite{nagurney2012projected}
\begin{align}\label{eq:vector_projection}
    \vproj{v}{x}{\mathcal{K}} = \lim_{\delta \rightarrow 0^+}\frac{\proj{x+\delta v}{\mathcal{K}} - x}{\delta}.
\end{align}

The following proposition gives a geometrical interpretation of the vector projection, which is used later to establish that the system under study can be indeed casted a projected dynamical system.

The interior and boundary of a set $\mathcal{A}$ is denoted by $\inter{\mathcal{A}}$ and $\bound{\mathcal{A}}$, respectively. 
\begin{proposition}[{\cite[Lemma 2.1]{nagurney2012projected}}] \label{prop:geometrical_interpretation_vector_projection}
    Let $n(x)$ denote the set of inward normals to $\mathcal{K}$ at $x$. 
    \begin{itemize}
        \item[(i).] If $x \in \inter{\mathcal{K}}$, then $\vproj{v}{x}{\mathcal{K}} = v$
        \item[(ii).] If $x \in \bound{\mathcal{K}}$, then $\vproj{v}{x}{\mathcal{K}} = v + \beta(x,v) n^*(x)$ with $n^*(x) = \arg \max_{n \in n(x)} v^\trans (-n)$ and $\beta(x,v) = \max\{0, v^\trans (-n^*(x))\}$.
    \end{itemize}
\end{proposition}
\begin{remark} \label{rem:geometrical_interpretation}
    It directly follows from Proposition~\ref{prop:geometrical_interpretation_vector_projection} case (ii) that if $x$ is at the boundary and $v$ points inwards into $\mathcal{K}$, $\vproj{v}{x}{\mathcal{K}} = v$ since $v^\trans (-n^*(x)) \leq 0$, which implies that $\beta(x,v) = 0$. Furthermore, if $x$ is at the boundary and $v$ points outwards and is perpendicular to $\mathcal{K}$, $\beta(x,v) n^*(x) = -v$ and $\vproj{v}{x}{\mathcal{K}} = 0$. 
\end{remark}

With the vector projection~\eqref{eq:vector_projection}, a special type of constrained, discontinuous system can be defined. Given a dynamical system $\dot{x} = \phi(x)$ and a closed and convex set $\mathcal{K}$, the projected dynamical system is defined as
\begin{align} \label{eq:projected_dynamical_sys}
    \dot{x} = \vproj{\phi(x)}{x}{\mathcal{K}}, \quad x(0) \in \mathcal{K}.
\end{align}

\begin{proposition}[{\cite[Th.~2.5]{nagurney2012projected}}] \label{prop:properties_projected_sys}
    Let $\phi: \R^n \rightarrow \R^n$ be Lipschitz on a closed and convex set $\mathcal{K}$. Then,
    \begin{itemize}
        \item[(i)] for any $x(0)\in\mathcal{K}$, there exists a unique solution $x(t)$ to the initial value problem~\eqref{eq:projected_dynamical_sys};
        \item[(ii)] if the sequence $x_m(0) \rightarrow x(0)$ as $m \rightarrow \infty$, then $x_m(t)$ converges to $x(t)$ uniformly on every compact set of $[0,\infty)$.  
    \end{itemize}
\end{proposition}

\subsection{Auxiliary proofs} \label{sec:appendix_proofs}

\begin{proposition} \label{prop:kernel_kronecker}
    Consider a matrix $A \in \R^{m \times n}$ with $m\geq n$, $\text{rank}(A) = n-1$ and a nullspace $\ker(A) = \{ \alpha 1_n \mid \alpha \in \R \}$. Then, the matrix 
    \begin{align} \label{eq:matrix}
        \left( A \otimes I_p \right) \in \R^{mp \times np} 
    \end{align}
    has a nullspace which is spanned by the vectors $1_n \otimes a_p$, where $a_p \in \R^p$ is an arbitrary constant vector. 
\end{proposition}

\begin{proof}
    Let $e_i$, $i \in \{ 1, \dots, n \}$ and $e_j$, $j \in \{ 1, \dots, p \}$ be a basis of $\R^n$ and $\R^p$, respectively. The vectors $e_i \otimes e_j$ form then a basis for $\R^n \otimes \R^p$, which contains the kernel of matrix~\eqref{eq:matrix}. To determine the kernel of~\eqref{eq:matrix}, we have to find vectors $v \in \R^n \times \R^p$ such that
    \begin{align} \label{eq:kernel}
        \left( A \otimes I_p \right) v = 0.
    \end{align}
    For that, we parametrize the vector $v$ using the basis as $v = \sum_{i=1}^n \sum_{j=1}^p v_{ij}(e_i \otimes e_j)$, with $v_{ij} \in \R$. Then,~\eqref{eq:kernel} reads 
    \begin{align} \label{eq:kernel_parametrized}
        \left( A \otimes I_p \right)\sum_{i=1}^n \sum_{j=1}^p v_{ij}(e_i \otimes e_j) = 0.
    \end{align}
    Using the mixed-product and associative properties P1 and P2 of the Kronecker product as in Section~\ref{sec:appendix_kronecker},~\eqref{eq:kernel_parametrized} can be rearranged to
    \begin{align} \label{eq:kernel_parametrized2}
        \sum_{i=1}^n \sum_{j=1}^p v_{ij} (A e_i \otimes e_j) = \sum_{j=1}^p \left( \sum_{i=1}^n v_{ij} A e_i \right) \otimes e_j = 0.
    \end{align} 
    Since the vectors $e_j$ in~\eqref{eq:kernel_parametrized2} are linearly independent, all the individual addends over index $j$ have to be zero in order to fulfill the equality. The only way to achieve this is if the left-hand side of each of the Kronecker products is zero, i.e.,
    \begin{align}\label{eq:kernel_parametrized3}
        \sum_{j=1}^p v_{ij} A e_i = A \sum_{j=1}^p v_{ij} e_i = 0.
    \end{align} 
    Equation~\eqref{eq:kernel_parametrized3} is true if and only if $\sum_{i=1}^n v_{ij} e_i$ is a nullvector of $A$, i.e., $\alpha_j 1_n$ for some $\alpha_j \in \R$. Thus, the vectors
    \begin{align}
        v = \sum_{j=1}^p \alpha_j 1_n \otimes e_j = 1_n \otimes \left( \sum_{j=1}^p \alpha_j e_j \right) = 1_n \otimes a_p 
    \end{align}
    span the kernel of $\left( A \otimes I_p \right)$. 
\end{proof}

\end{document}